\documentclass[11pt,reqno]{amsart}

\usepackage{amssymb, amsmath, amsthm, esint}
\usepackage{mathrsfs}
\usepackage{bbm,dsfont}
\usepackage{hyperref}
\usepackage{mathtools}
\usepackage{fullpage}
\usepackage{color}

\title[Discrete Stein--Wainger]{Discrete analogues of maximally modulated singular integrals
of Stein--Wainger type}

\author[B. Krause]{Ben Krause}
\address{BK: Department of Mathematics, King's College London, WC2R 2LS, UK}
\email{ben.krause@kcl.ac.uk}

\author[J. Roos]{Joris Roos}
\address{JR: Department of Mathematical Sciences, University of Massachusetts Lowell, USA\\
\& School of Mathematics, The University of Edinburgh, Scotland, UK}
\email{jroos.math@gmail.com}

\date{February 15, 2021}


\def\R{\mathbb{R}}
\def\N{\mathbb{N}}
\def\C{\mathbb{C}}
\def\Z{\mathbb{Z}}
\def\T{\mathbb{T}}
\def\Q{\mathbb{Q}}

\theoremstyle{plain}
\newtheorem{thm}{Theorem}[section]
\newtheorem{prop}[thm]{Proposition}
\newtheorem{lem}[thm]{Lemma}

\numberwithin{equation}{section}
\newcommand{\fm}[1]{{{#1}(\mathrm{D})}}

\begin{document}

\begin{abstract}
	Consider the maximal operator $$\mathscr{C} f(x) = \sup_{\lambda\in\mathbb{R}}\Big|\sum_{\substack{y\in\mathbb{Z}^n\setminus\{0\}}} f(x-y) e(\lambda |y|^{2d}) K(y)\Big|,\quad (x\in\mathbb{Z}^n),$$ where $d$ is a positive integer, $K$ a Calder\'on--Zygmund kernel and $n\ge 1$. This is a discrete analogue of a real-variable operator studied by Stein and Wainger. The nonlinearity of the phase introduces a variety of new difficulties that are not present in the real-variable setting. We prove $\ell^2(\mathbb{Z}^n)$--bounds for $\mathscr{C}$, answering a question posed by Lillian Pierce.
\end{abstract}
\maketitle

\section{Introduction}

Let $d$ and $n$ be positive integers and $K$ a homogeneous Calder\'on-Zygmund kernel on $\R^n$, taking the form
\[ K(x) = \mathrm{p.v.} \frac{\Omega(x)}{|x|^n}, \]
where $\Omega$ is a smooth function on $\R^n\setminus \{0\}$ that is homogeneous of degree zero.
We also assume that $\int_{\mathbb{S}^{n-1}} \Omega(x) d\sigma(x)=0$, where $\sigma$ denotes the surface measure on the sphere $\mathbb{S}^{n-1}\subset \R^n$.
Consider the following operator acting on functions $f:\Z^n\to \C$,
\begin{equation}\label{eqn:C-def}
\mathscr{C} f(x) = \sup_{\lambda\in\R}\Big|\sum_{y\in\Z^n\setminus\{0\}} f(x-y) e(\lambda |y|^{2d}) K(y)\Big|,\quad (x\in\Z^n),
\end{equation}
where $|y|=(y_1^2+\cdots+y_n^2)^{1/2}$ and $e(x)=e^{2\pi i x}$.
This is a discrete analogue of a maximal operator studied by Stein and Wainger \cite{SW01}. We also refer to $\mathscr{C}$ as a discrete Carleson operator. This is motivated by the formal resemblance to Carleson's operator given by the presence of a supremum over the modulation parameters $\lambda$. However, we stress that the (substantial) difficulties encountered in the analysis of the present operator are of a fundamentally different nature than those encountered in the analysis of Carleson's operator.
The nonlinearity of the phase causes a number of new challenges arising from a curious fusion of number--theoretic and analytic phenomena which are not present in the real--variable case. We refer to the introduction of \cite{KL17} for further discussions motivating the study of this operator and to \cite{MST15a}, \cite{MST15b} for background and recent progress on some other related discrete analogues in harmonic analysis. The following is our main result.
\begin{thm}\label{thm:main} There is a constant $C\in (0,\infty)$ such that
\begin{equation}\label{eqn:main}
\|\mathscr{C}f\|_{\ell^2(\Z^n)} \le C \|f\|_{\ell^2(\Z^n)}.
\end{equation}
The constant $C$ only depends on $d,n$ and $K$.
\end{thm}

The case $n=d=1$ was the subject of a question posed by Lillian Pierce during a 2015 workshop at the American Institute of Mathematics.
We build on key partial progress previously obtained in \cite{KL17}, where a restricted supremum was considered.
 
The specific choice of the phase in \eqref{eqn:C-def} and the assumptions made on the kernel $K$ are imposed primarily in favor of simplicity. 
Various extensions for other phase functions could be topics for further investigation.
Another interesting problem is to prove $\ell^p$ bounds for $p\not=2$, which will be a topic in a forthcoming sequel to this paper, \cite{KRfuture}.\\

{\emph{Structure of the paper.} In \S \ref{sec:pre} we introduce some basic facts and notations used throughout the proof. The most substantial of these are certain known exponential sum estimates from \cite{SW99}.\\ 
In \S \ref{sec:main} we give the proof of Theorem \ref{thm:main}. The basic strategy follows that of \cite{KL17}, splitting the multiplier into a number--theoretic approximate ('major arcs') and an error term ('minor arcs'). This approach goes back to Bourgain \cite{Bou89} and can be viewed as an instance of the Hardy--Littlewood circle method. The proof involves four distinct components, which (with a slight abuse of terminology) we refer to as `Minor arcs I/II' and 'Major arcs I/II'.\\
In \S \ref{sec:tts1} ('Minor arcs I') we perform a preliminary $TT^*$ argument to reduce the set of modulation parameters $\lambda$.\\
In \S \ref{sec:err} ('Minor arcs II') we estimate the error terms from a number--theoretic approximation of the multipliers. This is a standard argument using the fundamental theorem of calculus (which only becomes possible after the crucial reduction from \S \ref{sec:tts1}). This is already featured in \cite{KL17}.\\
In \S \ref{sec:wtt} ('Major arcs I') we handle the number--theoretic component of the main contribution to the multiplier by exploiting exponential sum estimates. A somewhat unanticipated dichotomy appears here between the cases $d=1$ and $d\ge 2$.\\
In \S \ref{sec:mft} ('Major arcs II') we handle the full supremum by combining the number--theoretic component with a delicate multi-frequency analysis similar to \cite{KL17}. A new aspect is that we make crucial use of a numerical inequality (see \eqref{eqn:rm}) that also appeared in recent works of Mirek and Trojan and Mirek, Stein and Trojan \cite{MT16}, \cite{MST15a}, \cite{MST15b}.

This allows us to avoid the use of versions of Bourgain's logarithmic multi-frequency lemma \cite{Bou89} and variation-norm estimates from \cite{GRY17}, which could be used to give an alternative argument (as suggested by a remark in \cite{KL17}). Avoiding the use of Bourgain's lemma is desirable in view of extensions beyond $\ell^2$.\\

{\em Acknowledgements.} J.R. is grateful to Shaoming Guo for many useful conversations about this problem.
The authors thank Shaoming Guo and Pavel Zorin-Kranich for pointing out an oversight in a previous preprint version of this paper, and 
the anonymous referee for a careful reading of the paper and numerous suggestions that have led to improvements of the exposition.

\section{Preliminaries}\label{sec:pre}

We write $A\lesssim B$ to denote existence of a constant $C$ such that $A\le C\cdot B$, where the admissible dependencies of the constant $C$ will be specified, or clear from context. Throughout the text we allow constants to depend on the ambient dimension $n$, the degree $d$ and the kernel $K$.
Similarly, $A\approx B$ signifies that both, $A\lesssim B$ and $B\lesssim A$. The notation $A=B+O(X)$ stands for $|A-B|\lesssim X$.
Moreover, we write $A\asymp B$ to express that $\frac12 B\le A\le 2B$

\subsection{\texorpdfstring{Fourier transforms on $\Z^n$, $\T^n$, $\R^n$ and transference}{Fourier transforms and transference}}
For Fourier transforms of functions $f:\Z^n\to \C$, $g:\T^n\to \C$ we use the notations
\[ \widehat{f}(\xi)=\mathcal{F}_{\Z^n}f (\xi)= \sum_{x\in\Z^n} e(-\xi\cdot x) f(x)\quad\text{and}\]
\[\mathcal{F}^{-1}[g](x) =\mathcal{F}_{\Z^n}^{-1}[g](x)= \int_{\T^n} e(\xi\cdot x) \widehat{g}(\xi) d\xi. \]
Here $\T^n = (\R/\Z)^n$. A function $g:\R^n\to \C$ that satisfies $g(x+z)=g(x)$ for all $z\in\Z^n$ will be called \emph{periodic} and be silently identified with the corresponding function on $\T^n$.\\
For a function $h:\R^n\to \C$ we write
\[ \widehat{h}(\xi) = \mathcal{F}_{\R^n}h(\xi)= \int_{\R^n} e(-\xi\cdot x) h(x) dx\quad\text{and}\]
\[\mathcal{F}^{-1}[h](x) = \mathcal{F}^{-1}_{\R^n} [h] (x) = \widehat{h}(-x). \]
In particular, Fourier transforms on $\Z^n$ or $\R^n$ will be denoted by the same symbols unless the distinction is not clear from context, or is emphasized for other reasons.

For a bounded periodic function $m:\R^n\to \C$ we denote by $\fm{m}$, the associated Fourier multiplier acting on $\Z^n$, defined as
\[ \fm{m}f(x) = \mathcal{F}_{\Z^n}^{-1}[ m\cdot \mathcal{F}_{\Z^n} f ](x),\quad (x\in\Z^n). \]
We slightly abuse notation and also write $\fm{m}$ for the Fourier multiplier acting on $\R^n$, defined as
\[ \fm{m}h(x) = \mathcal{F}^{-1}_{\R^n}[ m\cdot \mathcal{F}_{\R^n} f ](x),\quad (x\in\R^n). \]

Let $(m_\lambda)_{\lambda\in\Lambda}$ be a family of bounded functions supported on a fundamental domain of $\T^n$ (such as a translate of the unit cube $[0,1)^n$) and denote their periodizations by
\[ \mathbf{m}_\lambda(\xi) = \sum_{z\in\Z^n} m_\lambda(\xi+z),\quad(\xi\in\R^n). \]

We will make use of the following transference principle.

\begin{lem}\label{lem:pre-transf}
Suppose that for some constant $A>0$,
\[ \| \sup_{\lambda\in\Lambda} |\fm{m_\lambda} f| \|_{L^2(\R^n)} \le A \|f\|_{L^2(\R^n)}. \]
Then
\[ \| \sup_{\lambda\in\Lambda} |\fm{\mathbf{m}_\lambda} f| \|_{\ell^2(\Z^n)} \lesssim_n A \|f\|_{\ell^2(\Z^n)}, \]
where the implicit constant only depends on $n$. 
\end{lem}

The proof of this fact is standard (see \cite[Lemma 4.4]{Bou89}; there in the case $n=1$, but the argument also works also for $n\ge 2$).

\subsection{\texorpdfstring{Some notation and $TT^*$}{Some notation}} 
For a function $\mathcal{K}:\Z^n\times\Z^n\to \C$ we denote by $T_\mathcal{K}$ the operator defined formally by
\begin{equation}\label{eqn:pre-Tkdef}
T_\mathcal{K} f(x) = \sum_{y\in \Z^n} \mathcal{K}(x,y) f(y).
\end{equation}
Then the operator $T_\mathcal{K} T_\mathcal{K}^*$ is formally given by $T_\mathcal{K} T_\mathcal{K}^* = T_{\mathcal{K}^\sharp}$, where the kernel $\mathcal{K}^\sharp$ is
\[ \mathcal{K}^\sharp(x,y) = \sum_{z\in\Z^n} \mathcal{K}(x,z) \overline{\mathcal{K}(y,z)}.  \]

\subsection{Kernel decomposition}
Let $\psi$ be a smooth function on $\R^n$ supported in $\{1/2\le |x|\le 2\}$ with $0\le \psi\le 1$ and $\sum_{j\in\Z} \psi_j(x) = 1$ for every $x\not=0$, where $\psi_j(x)=\psi(2^{-j} x)$. Decompose
\[ K(x) = \sum_{j\ge 1} K_j (x),\]
with $K_1(x)=\sum_{j\le 1}\psi_j(x) K(x)$ and $K_j(x)=\psi_j(x)K(x)$ for $j\ge 2$.
Then for all $j\ge 1$ and all $x\in\R^n\setminus 0$,
\begin{equation}\label{eqn:pre-kernel}
|K_j(x)| \lesssim 2^{-jn},\quad|\nabla K_j(x)|\lesssim 2^{-j(n+1)},\quad \mathrm{supp}\;K_j\subset \{x\,:\,|x|\le 2^{j+1}\}.
\end{equation}

\subsection{A numerical inequality}

We record a Rademacher-Menshov-type numerical inequality that was also crucially used in \cite{MST15a} (Lemma 2.3 there): for complex numbers $(a_j)_{j=0,\dots,2^s}$ we have
\begin{equation}\label{eqn:rm}
\max_{0\le j\le 2^s} |a_j| \le |a_{j_0}| + \sqrt{2} \sum_{l=0}^s \Big(  \sum_{0\le \kappa <2^{s-l}} |a_{(\kappa+1)2^l} - a_{\kappa 2^l}|^2 \Big)^{1/2},
\end{equation}
for every integer $j_0$ with $0\le j_0\le 2^s$. This follows from an appropriate decomposition of the interval $[0,2^s]$ into dyadic intervals, see \cite{LL,MT16}.

\subsection{Exponential sum estimates}
Given integers $x_1,x_2,\dots,x_m$ at least one of which is non--zero we often use the notation $(x_1,x_2,\dots,x_m)$ for the greatest common divisor of $x_1, \dots, x_m$. It will be clear from context whether $(x_1,\dots,x_m)$ refers to the greatest common divisor, or the vector of the integers $x_1,\dots,x_m$.
For a positive integer $q$ we use the notation
\[ [q] = \Z \cap [0,q). \]
The letter $q$ always denotes a positive integer throughout the text. By a \emph{reduced rational} we mean a fraction $\tfrac{a}q$ with $a\in \Z$ and $(a,q)=1$. For a positive integer $D\ge 2$, $x\in\R^n$ and real coefficients $\xi=(\xi_\alpha)_{1\le|\alpha|\le D}$ we define the polynomial
\[ P(\xi; x) = \sum_{1\le|\alpha|\le D} \xi_\alpha x^\alpha, \]
where $\alpha\in \N_0^n$ denotes a multiindex. A key ingredient will be the following exponential sum estimate, due to Stein and Wainger \cite[Proposition 3]{SW99}.

\begin{prop}\label{prop:weylpower}
	Let $R\ge 1$, $\varphi$ a smooth function on $\R^n$ such that $|\varphi(x)|\le 1$ and $|\nabla \varphi(x)|\le (1+|x|)^{-1}$ for all $x\in\R^n$,
	and $\omega$ a convex set contained in the ball of radius $100 R$ centered at the origin. 
	Then for every $\varepsilon>0$ there exists $\delta>0$ only depending on $\varepsilon,n,D$ 
	such that the following holds: 
	for every $\xi$ with the property that 
	for some $\alpha_0$ with $1\le |\alpha_0|\le D$ there exists 
	a reduced rational $\tfrac{a}{q}\in \Q$ such that
	\[ |\xi_{\alpha_0} - \tfrac{a}q| \le \tfrac1{q^2}\quad\text{and}\quad
	R^\varepsilon \le q \le R^{|\alpha_0|-\varepsilon}, \]
	we have
	\[  \Big|\sum_{x\in\Z^n\cap \omega} e(P(\xi; x))
	\varphi(x)\Big| \le C R^{n-\delta}, \]
	where the constant $C$ only depends on $\varepsilon$, $n$, $D$.
\end{prop}

\subsection{Approximation of the multipliers}

For $j\ge 1$, $\lambda\in\R$ and $\xi\in\R^n$ we define the multipliers
\begin{equation}\label{eqn:mjdef}
m_{j,\lambda}(\xi) = \sum_{y\in\Z^n} e(\lambda |y|^{2d} + \xi\cdot y) K_j(y).
\end{equation}
This defines a periodic function both in $\lambda$ and $\xi$. Following Bourgain \cite{Bou89}, the starting point for our arguments is an appropriate approximation for the value of $m_{j,\lambda}(\xi)$ when $\xi$ and $\lambda$ are close to rationals with small denominator. To formulate the result, we define the exponential sums
\begin{equation}\label{eqn:defcompleteweylsum} S(\tfrac{a}{q},\tfrac{\mathbf{b}}{q}) = \frac{1}{q^n} \sum_{r\in [q]^n} e(\tfrac{a}{q} |r|^{2d}+\tfrac{\mathbf{b}}{q}\cdot r)
\end{equation}
for rationals $\tfrac{a}{q}\in\Q$, $\tfrac{\mathbf{b}}{q}\in\Q^n$ with $(a,\mathbf{b},q)=1$ (note that this condition makes $S(\tfrac{a}{q},\tfrac{\mathbf{b}}{q})$ well--defined). 

By Proposition \ref{prop:weylpower} there exists $\delta>0$ so that
\begin{equation}\label{eqn:Sdecay}
|S(\tfrac{a}{q},\tfrac{\mathbf{b}}{q})| \lesssim_{d,n} q^{-\delta}.
\end{equation}
The following observation will be crucial at various points in the proof of Theorem \ref{thm:main}.

\begin{lem}\label{lem:weylorth}
Suppose that $\frac{a}q\in \Q, \frac{\mathbf{b}}q\in \Q^n$, $(a,\mathbf{b},q)=1$ and $(a,q)>1$. Then $S(\tfrac{a}{q}, \tfrac{\mathbf{b}}{q})~=~0$.
\end{lem}
We postpone the standard proof of this to the end of this section. Next, we define the real--variable versions of the multipliers $m_{j,\lambda}(\xi)$ by
\begin{equation}\label{eqn:Phijdef}
\Phi_{j,\lambda}(\xi) = \int_{\R^n} e(\lambda |y|^{2d} + \xi\cdot y)K_j(y) dy.
\end{equation}
At this point we record the following standard oscillatory integral decay estimate in the spirit of van der Corput's lemma:
\begin{equation}\label{eqn:Phij-est}
|\Phi_{j,\lambda}(\xi)| \lesssim (1+2^{2dj}|\lambda|+2^{j}|\xi|)^{-\frac1{2d}}.
\end{equation}
For the proof we refer to \cite[Proposition 2.1]{SW01}. This estimate does not enter in the proof of the approximation result in this section, but will be important later on.
Our basic approximation result for the multipliers $m_{j,\lambda}(\xi)$ now reads as follows.
\begin{lem}\label{lem:approx}
	Let $j,q$ be positive integers with $q\le 2^{j-2}$. Let $a\in\Z, \mathbf{b}\in \Z^n$ with $(a,\mathbf{b},q)=1$.
	Further, assume that $\lambda\in\R$, $\xi\in\R^n$ are such that
	\begin{equation}\label{prop:approx-asm}
	|\lambda-\tfrac{a}{q}|\le \delta 2^{-(2d-1)j} \quad\text{and}\quad|\xi-\tfrac{\mathbf{b}}{q}|\le \delta, 
	\end{equation}
	where $\delta\in (2^{-j},1)$. Then
	\begin{equation}\label{prop:approx-con}
	m_{j,\lambda}(\xi) = S(\tfrac{a}{q}, \tfrac{\mathbf{b}}{q})\Phi_{j,\lambda - \tfrac{a}{q}}(\xi-\tfrac{\mathbf{b}}{q}) + O(q\delta),
	\end{equation}
	where the implicit constant depends only on $d,n,K$.
\end{lem}
The proof is similar to that of the corresponding statement in \cite{Bou89} (see Lemma 5.12 there).

\begin{proof}[Proof of Lemma \ref{lem:approx}]
Writing $y=u q + r$ with $u\in\Z^n$, $r\in [q]^n$, we can express $m_{j,\lambda}(\xi)$ as
\[ q^{-n} \sum_{r\in [q]^{n}} e(\tfrac{a}q |r|^{2d} + \tfrac{\mathbf{b}}{q}\cdot r) I_{q,r}(\lambda-\tfrac{a}{q}, \xi-\tfrac{\mathbf{b}}{q}),  \]
where
\[ I_{q,r}(\nu, \eta) = q^n \sum_{u\in\Z^n} e (\nu|uq+r|^{2d} + \eta \cdot (uq+r)) K_j(uq+r). \]
It suffices to show that for every $r\in [q]^n$ and every $(\nu,\eta)\in \R\times\R^{n}$ with
\[ |\nu|\le \delta 2^{-(2d-1)j},\quad |\eta|\le \delta \]
we have the relation
\begin{equation}\label{eqn:approx-claim}
I_{q,r}(\nu,\eta) = \int_{\R^n} e(\nu |t|^{2d} + \eta\cdot t) K_j(t) dt + O(\delta q).
\end{equation}
The integral on the right--hand side of \eqref{eqn:approx-claim} equals
\[ q^n \int_{\R^n} e(\nu |tq+r|^{2d} + \eta (tq+r)) K_j(tq+r) dt, \]
which in turn can be split as
\begin{equation}\label{eqn:approx-pf1}
q^n \sum_{u\in \Z^n}  \int_{[0,1]^n} e(\nu |uq+r+tq|^{2d}+\eta\cdot (uq+r+tq) ) K_j(uq+r+tq) dt.
\end{equation}
In this display it holds that
\[ |\nu |uq+r+tq|^{2d} - \nu |uq+r|^{2d}|\lesssim \delta q \]
since $|r|\le q$, $|uq+r+qt|\approx |uq+r|\approx 2^j$ and $\nu\le \delta 2^{-(2d-1)j}$.
Similarly,
\[ |\eta\cdot (uq+r+tq) - \eta\cdot (uq+r)| \lesssim \delta q. \]
Using also that $\int_{\R^n} |K_j(t)| dt\approx 1$, this yields that \eqref{eqn:approx-pf1} is 
\begin{equation}\label{eqn:approx-pf2}
q^n \sum_{u\in \Z^n}  \int_{[0,1]^n} e(\nu |uq+r|^{2d}+\eta\cdot (uq+r) ) K_j(uq+r+tq) dt + O(\delta q).
\end{equation}
Finally, note from \eqref{eqn:pre-kernel} that
\[ |K_j(uq+r+tq) - K_j(uq+r)| \lesssim 2^{-j(n+1)} q\le 2^{-jn} \delta q. \]
Then we see that \eqref{eqn:approx-pf2} can be written as
\[ q^n \sum_{u\in \Z^n} e(\nu |uq+r|^{2d}+\eta\cdot (uq+r) ) K_j(uq+r) + O(\delta q), \]
which establishes \eqref{eqn:approx-claim}.
\end{proof}

\begin{proof}[Proof of Lemma \ref{lem:weylorth}]
Let $(a,q)=v>1$. Write $a=a'v$ and $q=q'v$. Then
	
	\[ q^n S(\tfrac{a}q,\tfrac{\mathbf{b}}{q}) = \sum_{u\in [v]^n} \sum_{r\in [q']^n} e(\tfrac{a'}{q'} |u q'+r|^{2d} + \tfrac{\mathbf{b}}{q} \cdot (u q'+r)) \]
	\[ =\left[ \sum_{r\in [q']^n} e(\tfrac{a'}{q'} |r|^{2d} + \tfrac{\mathbf{b}}{q}\cdot r)\right]
	\prod_{i=1}^n \sum_{u_i\in [v]} e(\tfrac{\mathbf{b}_i}{v} \cdot u_i) \]
	Since $(a,\mathbf{b},q)=1$ and $v>1$, there must exist $i_0$ such that $\mathbf{b}_{i_0}$ is not divisible by $v$. But that implies $\sum_{\ell\in [v]} e(\tfrac{\mathbf{b}_{i_0}}{v} \ell) = 0.$
\end{proof}

\section{Proof of Theorem \ref{thm:main}}\label{sec:main}

To prove the theorem, we need to obtain an $\ell^2(\Z^n)$ bound for the maximal operator
\[ \sup_{\lambda\in\R} \left|\sum_{j\ge 1} \fm{m_{j,\lambda}} f\right|,  \]
where $m_{j,\lambda}$ is defined in \eqref{eqn:mjdef}.
A first observation is that for each fixed $j$,
\[ \| \sup_{\lambda\in\R} |\fm{m_{j,\lambda}}f| \|_{\ell^2(\Z^n)} \lesssim \|f\|_{\ell^2(\Z^n)}, \]
by the triangle inequality, Young's convolution inequality and \eqref{eqn:pre-kernel}. As a consequence, we may in the following assume that $j\ge j_0$, where $j_0$ is a sufficiently large constant depending on $d$ and $n$.

Before we proceed, we give a rough description of what will be done. 
For this purpose, we will be deliberately vague when using the terms
'small' and 'close'. At this point,
the reader should imagine these terms as being relative to appropriate fractional 
powers of $2^j$, which might differ at each occurrence and will have to be chosen carefully in the sequel. 
Roughly speaking, the approximation \eqref{prop:approx-con} tells us what $m_{j,\lambda}(\xi)$ is when $\lambda$ and $\xi$ are close to rationals with small denominator. 
On the other hand, Proposition \ref{prop:weylpower} tells us that $|m_{j,\lambda}(\xi)|$ is small if any of $\lambda, \xi_1,\dots,\xi_n$ is not close to a rational with small denominator. 
This naturally leads to a decomposition of $m_{j,\lambda}$ into two new functions. The first arises from summing the main contributions 
$S(\tfrac{a}{q},\tfrac{\mathbf{b}}{q})\Phi_{j,\lambda-\tfrac{a}{q}}(\xi-\tfrac{\mathbf{b}}{q})$ over a suitable collection of rational $(\tfrac{a}{q},\tfrac{\mathbf{b}}{q})$ with small $q$. In the terminology of the Hardy--Littlewood circle method, these are the \emph{major arcs}. The second function is an error term, which will subsume both the approximation error from \eqref{prop:approx-con} and the \emph{minor arcs}, i.e. the cases when at least one of $\lambda,\xi_1,\dots,\xi_n$ is not close to one of the chosen rationals. This decomposition is stated below as \eqref{eqn:Ejdef}. 
Following this approach naively already leads to a fundamental problem: the error term crucially depends on $\lambda$, 
but we know only little more about it except that its absolute value is small. 
This leaves us with few strategies to handle the maximal operator corresponding to the error term. 
This was one of the reasons for the restriction on the parameters $\lambda$ imposed in \cite{KL17}. 
By a preliminary $TT^*$ argument on the multiplier $m_{j,\lambda}(\xi)$, we may discard 'most' parameters $\lambda$: 
as long as we discard $\lambda$ sufficiently close to a rational with sufficiently small denominator, 
the $TT^*$ argument yields summable decay in $j$ (see Proposition \ref{prop:tts1} below).
For each $j$, this only leaves $\lambda$ contained in a union of a few small intervals (see \eqref{eqn:pre-Xjdef} below). 
This allows us to bound the remaining maximal operator for the error term by a standard argument using 
the fundamental theorem of calculus, 
the crucial size information on the error and a crude $\lambda$--derivative estimate (see Proposition \ref{prop:minor} below). We proceed with the precise estimates. 

\subsection{Decomposition of the multiplier and minor arcs}
Define
\begin{equation}\label{eqn:pre-Ajdef}
\mathfrak{A}_j = \{ \tfrac{a}{q}\in\Q\,:\,(a,q)=1,\, q\in \Z\cap [1,2^{\lfloor j\varepsilon_1\rfloor}) \},
\end{equation}
\begin{equation}\label{eqn:pre-Xjdef}
X_j = \bigcup_{\alpha\in\mathfrak{A}_j} \{ \lambda\in\R\,:\,|\lambda-\alpha|\le 2^{-2dj+\varepsilon_1j} \},
\end{equation}
where $\varepsilon_1 \in (0, 2^{-5})$ is a small fixed number that will be determined depending on $d$ and $n$. Observe that the union in \eqref{eqn:pre-Xjdef} is disjoint.
The $TT^*$ argument alluded to above yields the following result.
\begin{prop}\label{prop:tts1}There exists $\gamma>0$ only depending on $d$, $n$, $\varepsilon_1$ such that for all $j\ge 1$,
\[ \| \sup_{\lambda\not\in X_j} |\fm{m_{j,\lambda}}f| \|_{\ell^2(\Z^n)} \lesssim 2^{-j\gamma} \|f\|_{\ell^2(\Z^n)}. \]
\end{prop}
The proof can be seen as somewhat parallel to that of Stein--Wainger \cite{SW01} and is given in \S \ref{sec:tts1}.
From now on we can restrict our attention to the multipliers
$m_{j,\lambda}(\xi)\mathbf{1}_{X_j}(\lambda)$.
In order to define the major arc approximations we need to set up some notation. For a positive integer $s$ define
\[ \mathcal{R}_s = \{(\tfrac{a}{q},\tfrac{\mathbf{b}}{q})\in\Q\times\Q^n\,:\,(a,\mathbf{b},q)=1,\, q\in \Z\cap [2^{s-1},2^s)\}. \]
Fix a smooth radial function $\chi$ on $\R^n$  with $0\le\chi\le 1$ that is supported in $\{|\xi|\le 1/2\}$ and equal to one on $[-1/4,1/4]^n$. For $s\ge 1$ and $\xi\in\R^n$ we write $\chi_s(\xi)=\chi(2^{10s}\xi)$. Further define for $s$ with $s\le \varepsilon_1 j$,
\begin{equation}\label{eqn:Ljsdef}
L^s_{j,\lambda}(\xi) = \sum_{(\alpha,\beta)\in \mathcal{R}_s} S(\alpha,\beta) \Phi^*_{j,\lambda-\alpha}(\xi-\beta)\chi_s(\xi-\beta),
\end{equation}
where $\Phi^*_{j,\nu}$ is given by
\begin{equation}\label{eqn:Phijstar-def}
\Phi^*_{j,\nu} = \Phi_{j,\nu} \cdot \mathbf{1}_{|\nu|\le 2^{-2dj+\varepsilon_1 j}}.
\end{equation}
From the definition of $\mathcal{R}_s$ it is clear that $L_{j,\lambda}^s(\xi)$ is periodic in $\lambda$ and $\xi$.
Also note that if $L_{j,\lambda}^s(\xi)\not=0$ (where $s\le \varepsilon_1 j$), then $\lambda\in X_j$. Define
\begin{equation}\label{eqn:Ljdef}
L_{j,\lambda} = \sum_{1\le s\le \varepsilon_1 j} L^s_{j,\lambda}.
\end{equation}
Next, the function $E_{j,\lambda}$ is defined as the difference of $m_{j,\lambda} \mathbf{1}_{X_j}(\lambda)$ and $L_{j,\lambda}$ so that
\begin{equation}\label{eqn:Ejdef}
m_{j,\lambda}\cdot \mathbf{1}_{X_j}(\lambda) = L_{j,\lambda} + E_{j,\lambda}.
\end{equation}
From the definitions, $L_{j,\lambda}(\xi)$ and $E_{j,\lambda}(\xi)$ are periodic in $\lambda$ and $\xi$ and vanish unless $\lambda\in X_j$.
\begin{prop}\label{prop:minor}If the constant $\varepsilon_1$ is chosen small enough 
(depending only on $d$ and $n$),
there exists $\gamma>0$ depending on $d$, $n$, $\varepsilon_1$ such that for all $j\ge 1$,
\[ \| \sup_{\lambda\in X_j} |\fm{E_{j,\lambda}}f| \|_{\ell^2(\Z^n)} \lesssim 2^{-j\gamma} \|f\|_{\ell^2(\Z^n)}. \]
\end{prop}
The proof is given in \S \ref{sec:err}. The basic idea is that the absolute value of $E_{j,\lambda}$ should be small (two reasons to believe this are Lemma \ref{lem:approx} and Proposition \ref{prop:weylpower}) and its $\lambda$--derivatives are not too large. The structure of $X_j$ then allows us to effectively deploy the fundamental theorem of calculus to deal with the supremum over $\lambda$.

\subsection{Major arcs}

It now remains to bound the maximal operator associated with the multiplier
\[ \sum_{j\ge 1} L_{j,\lambda} = \sum_{j\ge 1} \sum_{1\le s\le \varepsilon_1 j} L^s_{j,\lambda} = \sum_{s\ge 1} L^s_{\lambda}, \]
where we have set
\begin{equation}\label{eqn:Lsdef}
L^s_\lambda = \sum_{j\ge \varepsilon_1^{-1} s} L^s_{j,\lambda}.
\end{equation}
The proof of Theorem \ref{thm:main} will be completed if we can exhibit $\gamma>0$ such that for all $s\ge 1$,
\begin{equation}\label{eqn:major-main}
\|\sup_{\lambda\in\R} |\fm{L^s_{\lambda}}f| \|_{\ell^2(\Z^n)} \lesssim_{d,n} 2^{-\gamma s} \|f\|_{\ell^2(\Z^n)}.
\end{equation}
We now begin with the definition of some auxiliary sets of rationals:
\[ \mathcal{A}_s = \{ \alpha\in\Q \,:\, (\alpha,\beta)\in \mathcal{R}_s\;\text{for some}\;\beta \}, \]
\[ \mathcal{B}_s(\alpha) = \{ \beta\in\Q^n\,:\,(\alpha,\beta)\in\mathcal{R}_s \}, \]
\begin{equation}\label{eqn:Bsharp-def}
\mathcal{B}_s^\sharp = \{ \tfrac{\mathbf{b}}{q}\,:\, \mathbf{b}\in\Z^n,\, q\in \Z\cap [2^{s-1},2^s) \}.
\end{equation}
By definition,
\[ (\alpha,\beta)\in\mathcal{R}_s\;\Longleftrightarrow\; \alpha\in \mathcal{A}_s,\, \beta\in \mathcal{B}_s(\alpha) \]
and
\[ \mathcal{B}_s(\alpha) \subset \mathcal{B}_s^\sharp\quad\text{for all}\;\alpha. \]
Also note that $\mathcal{B}_s(\alpha)=\emptyset$ if $\alpha\not\in \mathcal{A}_s$. 
Fix a smooth radial function $\widetilde{\chi}$ with $0\le \widetilde{\chi}\le 1$ that equals 
to one on $\{|\xi|\le 1/2\}$ 
(and hence on the support of $\chi$) and is supported in $\{|\xi|\le 1\}$. 
Set $\widetilde{\chi}_s(\xi)=\widetilde{\chi}(2^{10s}\xi)$. 
Given a bounded function $m$ on $\R^n$ we define the periodic multipliers
\begin{equation}\label{eqn:Lsmdef}
\mathscr{L}_{s,\alpha}[m](\xi) = \sum_{\beta\in\mathcal{B}_s(\alpha)} S(\alpha,\beta) m(\xi-\beta) \chi_s(\xi-\beta),
\end{equation}
\begin{equation}\label{eqn:Lsmshdef}
\mathscr{L}^\sharp_{s}[m](\xi) = \sum_{\beta\in\mathcal{B}^\sharp_s} m(\xi-\beta) \widetilde{\chi}_s(\xi-\beta).
\end{equation}
A crucial observation is the factorization
\begin{equation}\label{eqn:Lfact}
\mathscr{L}_{s,\alpha}[m] = \mathscr{L}_{s,\alpha}[1]\cdot \mathscr{L}^\sharp_s[m],
\end{equation}
which holds because for each $\xi$, there is at most one $\beta\in\mathcal{B}_s^\sharp$ so that $\widetilde{\chi}_s(\xi-\beta)\not=0$.
The kernel associated with the multiplier \eqref{eqn:Lsmdef} is given by
\[
\mathcal{F}_{\Z^n}^{-1}[\mathscr{L}_{s,\alpha}[m]](y) =  \sum_{\beta\in\mathcal{B}_s(\alpha)} S(\alpha,\beta) \int_{[0,1]^n} e(\xi\cdot y) m(\xi-\beta) \chi_s(\xi-\beta) d\xi
\]
\begin{equation}\label{eqn:Lsmkernel}
= \sum_{\beta\in \mathcal{B}_s(\alpha)\cap [0,1)^n} S(\alpha,\beta) e(\beta\cdot y) \mathcal{F}_{\R^n}^{-1} [m\cdot \chi_s] (y),
\end{equation}
where $y\in\Z^n$.
With this notation in mind we write the multiplier in question as
\begin{equation}\label{eqn:Lslambda}
 L^s_\lambda = \mathscr{L}_{s,\alpha}[\Phi^s_{\lambda-\alpha}], 
\end{equation} 
where we have set
\begin{equation}\label{eqn:Phis-def}
\Phi^s_\lambda=\sum_{j\ge \varepsilon_1^{-1}s} \Phi^*_{j,\lambda}.
\end{equation}
and $\alpha$ is the unique element of $\mathcal{A}_s$ such that $|\lambda-\alpha|\le 2^{-2s-10}$ (say), 
or an arbitrary value from the complement of $\mathcal{A}_s$ if no such $\alpha$ exists 
(in this case, $L^s_\lambda(\xi)=0$ anyways). 
Here, uniqueness of such $\alpha$ follows because two distinct rationals 
with denominators $\le 2^s$ must be at least $2^{-2s}$ apart. 
In view of \eqref{eqn:Lslambda} and the factorization \eqref{eqn:Lfact} it 
is reasonable to begin with the following number-theoretic estimate. 
\begin{prop}\label{prop:weylsums}
There exists $\gamma>0$ depending on $d,n$ such that for every $s\ge 1$
\begin{equation}\label{eqn:weylsums-main}
\| \sup_{\alpha\in\mathcal{A}_s} |\fm{\mathscr{L}_{s,\alpha}[1]}f| \|_{\ell^2(\Z^n)} \lesssim 2^{-\gamma s} \|f\|_{\ell^2(\Z^n)}.
\end{equation}
\end{prop}
This will be proved in \S \ref{sec:wtt} by making use of exponential sum estimates.
The factorization \eqref{eqn:Lfact} invites us to consider
the companion maximal operator 
\[f\mapsto \sup_{\mu\in\R} |\fm{\mathscr{L}^\sharp_s[\Phi_\mu^s]} f|.\]
Using Bourgain's multi-frequency lemma and the variational estimates from \cite{GRY17} it
is possible to show that this maximal operator has $\ell^2\to\ell^2$ operator norm $\lesssim s^2$ (the proof is omitted in this paper, because this claim will not be needed).
However, it is technically not straightforward to combine this result with Proposition \ref{eqn:weylsums-main} to treat the
maximal operator associated with \eqref{eqn:Lslambda}.
Instead, we take a different route that relies on the numerical inequality \eqref{eqn:rm} and a theorem of Stein and Wainger \cite{SW01}.
The following proposition is proved in \S \ref{sec:mft}.
\begin{prop}\label{prop:multifreq}
The constant $\varepsilon_1$ can be chosen small enough depending on $d$ and $n$ so that
there exists $\gamma>0$ depending on $d,n$ such that for every $s\ge 1$,
\begin{equation}\label{eqn:multifreq-main}
\| \sup_{\lambda\in\R} | \fm{L^{s}_{\lambda}} f| \|_{\ell^2(\Z^n)} \lesssim 2^{-\gamma s} \|f\|_{\ell^2(\Z^n)}.
\end{equation}
\end{prop}
This establishes \eqref{eqn:major-main} and thereby Theorem \ref{thm:main}.

\section{Minor arcs I: Proof of Proposition \ref{prop:tts1}}\label{sec:tts1}
Since the output $\fm{m_{j,\lambda}}f(x)$ only depends on the values of $f$ in a $2^{j+1}$--neighborhood of the point $x$, a standard localization argument allows us to assume that $f$ is supported in the set $B_j=\{y\in\Z^n\,:\,|y|\le 2^j \}$. Fix an arbitrary function $\lambda:\Z^n\to \R\setminus X_j$ and write
\[ T_{j,\lambda} f(x) = \fm{m_{j,\lambda(x)}} (f \mathbf{1}_{B_j})(x) = \sum_{y\in \Z^n} f(y) \mathcal{K}_{j,\lambda}(x,y),  \]
where
\[ \mathcal{K}_{j,\lambda}(x,y) = e(\lambda(x)|x-y|^{2d}) K_j(x-y)\mathbf{1}_{B_j}(y). \]
Then the kernel of $T_{j,\lambda}T^*_{j,\lambda}$ is given by
\begin{equation}\label{eqn:tts1-ker}
\mathcal{K}^\sharp_{j,\lambda}(x,y) = \sum_{z\in \Z^n} e(\lambda(x)|z|^{2d} - \lambda(y)|y-x+z|^{2d})
\end{equation}
\[\hspace{4cm}\times K_j(z)\overline{K_j(y-x+z)}\mathbf{1}_{B_j}(x-z). \]
Note that $\mathcal{K}^\sharp_{j,\lambda}(x,y)=0$ unless
\begin{equation}\label{eqn:tts1-kersupp}
|x|\le 2^{j+2}\quad\text{and}\quad |y|\le 2^{j+2}.
\end{equation}
Let $\delta_0>0$ and $c_0>0$ be determined later and define
\[ E_{j,\lambda} = \{ (x,y)\in \Z^n\times\Z^n \,:\, |\mathcal{K}^\sharp_{j,\lambda}(x,y)|\ge c_0 2^{-j(n+\delta_0)} \}.  \]
\begin{lem}\label{lem:tts1-main}
The constants $c_0$ and $\delta_0$ can be chosen depending on $d,n,\varepsilon_1$ such that for every $j\ge 1$ it holds that
\begin{equation}\label{eqn:tts1-Esmall}
|E_{j,\lambda}| \lesssim 2^{2nj-\frac1{10}\varepsilon_1 j}.
\end{equation}
where $\varepsilon_1$ is as in \eqref{eqn:pre-Ajdef}, \eqref{eqn:pre-Xjdef}.
\end{lem}

Before proving this statement we show how it can be used to finish the proof of Proposition \ref{prop:tts1}. By definition of $E_{j,\lambda}$,
\[ |\mathcal{K}^\sharp_{j,\lambda}(x,y)| \lesssim 2^{-nj-\delta_0j} \mathbf{1}_{B_{j+2}\times B_{j+2}}(x,y) + 2^{-nj}\mathbf{1}_{E_{j,\lambda}}(x,y). \]
With \eqref{eqn:tts1-Esmall} this implies
\begin{equation}\label{eqn:tts1-K2small}
\|\mathcal{K}^\sharp_{j,\lambda}\|_{\ell^2(\Z^n\times\Z^n)} \lesssim 2^{-\delta_0 j}+2^{-\frac1{20}\varepsilon_1 j}.
\end{equation}
By the Cauchy--Schwarz inequality we have
\[ |\langle T_{\mathcal{K}^\sharp_{j,\lambda}} f, g\rangle| \le \sum_{x\in\Z^n}\sum_{y\in\Z^n} |g(x)| |f(y)| |\mathcal{K}^\sharp_{j,\lambda}(x,y)|\le \|f\|_{\ell^2(\Z^n)} \|g\|_{\ell^2(\Z^n)} \|\mathcal{K}^\sharp_{j,\lambda}\|_{\ell^2(\Z^n\times\Z^n)},  \]
which by \eqref{eqn:tts1-K2small} and $\ell^2$ duality leads to
\[ \|T_{j,\lambda}\|_{\ell^2\to \ell^2} = \|T_{\mathcal{K}^\sharp_{j,\lambda}}\|_{\ell^2\to \ell^2}^{1/2} \lesssim 2^{-\gamma j} \]
with $\gamma=\min(\frac12 \delta_0, \frac1{40} \varepsilon_1)$. It remains to prove Lemma \ref{lem:tts1-main}. 

In fact we will prove something stronger: the claim is that after choosing $c_0$ and $\delta_0$ suitably, we have for every fixed $(x',y^*)\in \Z^{n-1}\times \Z^n$ that
\begin{equation} \label{eqn:tts1-claim}
|\{x_1\in\Z\,:\,(x_1,x',y^*)\in E_{j,\lambda} \}| \lesssim 2^{j-\frac1{10}\varepsilon_1 j}.
\end{equation}
In other words, each $(x',y^*)$--slice of $E_{j,\lambda}$ has small cardinality. 
By Fubini's theorem and \eqref{eqn:tts1-kersupp} this implies the claimed inequality \eqref{eqn:tts1-Esmall}.

For future reference, we will be more careful with explicit constants than strictly necessary in this proof. 
The reader can safely ignore all constants only depending on $d$ in the estimates that follow. 
Fixing $(x',y^*)\in \Z^{n-1}\times \Z^n$, we define
\[ \mathcal{E} = \{ x_1\in\Z\,:\, (x_1,x',y^*)\in E_{j,\lambda} \}. \]
Set $\varepsilon_0 = \frac1{10}\varepsilon_1$.\\
{\emph{Claim.}} The numbers $c_0$ and $\delta_0$ can be chosen such that the following holds: for every $u\in \mathcal{E}$ there exists a reduced rational $\tfrac{a}q$ with $q\le 2^{\varepsilon_0 j + 1} d$ such that
\begin{equation}\label{eqn:tts1-claim1}
|(u-y^*_1)\lambda(y^*) - \tfrac{a}q| \le 2^{-j(2d-1) + \varepsilon_0 j}
\end{equation}

\begin{proof}
	Note that the coefficient of $z_1^{2d-1}$ in the phase of \eqref{eqn:tts1-ker} is equal to $2d(x_1-y_1) \lambda(y)$. By Dirichlet's approximation theorem, there exists a reduced rational $\tfrac{a}q$ with $q\le 2^{j(2d-1)-\varepsilon_0 j}$ such that
	\[ |2d(u-y_1^*)\lambda(y^*) - \tfrac{a}q| \le q^{-1} 2^{-j(2d-1)+\varepsilon_0 j}\le \tfrac1{q^2}. \]
	Applying Proposition \ref{prop:weylpower} (with $R=2^{j}$) we may choose $c_0$ and $\delta_0$ (depending on the choice of $\varepsilon_0$) so that $q\le 2^{\varepsilon_0 j}$ (because $|\mathcal{K}^\sharp_{j,\lambda}(u,x',y^*)|\ge c_0 2^{-j(n+\delta_0)}$).
	Dividing through by $2d$ yields the claim.
\end{proof}

From now on we fix $c_0$ and $\delta_0$ to make the statement in the claim valid. We will also assume $j\ge j_0$, where $j_0$ is a large constant depending only on $d$ that will be determined later. Our goal is now to show that $|\mathcal{E}| \le 2^{j-\varepsilon_0 j}$. Arguing by contradiction, we assume that
\begin{equation}\label{eqn:tts1-asm}
|\mathcal{E}| > 2^{j-\varepsilon_0 j}.
\end{equation}
It is clear that 
\begin{equation}\label{eqn:tts1-triv}
\mathcal{E}\subset [-2^{j+2}, 2^{j+2}].
\end{equation}

We now exploit the three properties \eqref{eqn:tts1-claim1}, \eqref{eqn:tts1-asm}, \eqref{eqn:tts1-triv} to prove that $\lambda(y^*)\in X_j$, which establishes the required contradiction.
First, we claim that there exist $u_1,u_2\in\mathcal{E}$ such that
\begin{equation}\label{eqn:tts1-claim2}
1\le u_2-u_1\le 2^{\varepsilon_0 j+5},
\end{equation}
Indeed, suppose that all elements of $\mathcal{E}$ were pairwise separated by at least $2^{\varepsilon_0j+5}$. Then, by \eqref{eqn:tts1-triv} we would have $|\mathcal{E}|\le 2^{j-\varepsilon_0j-1}$, which contradicts \eqref{eqn:tts1-asm}. Consequently, there must exist $u_1,u_2\in\mathcal{E}$ such that \eqref{eqn:tts1-claim2} holds. 
By \eqref{eqn:tts1-claim1} there exist reduced rationals $\tfrac{a}{q}, \tfrac{a'}{q'}$ with $\max(q,q')\le 2^{\varepsilon_0j+1} d$ and 
\[ |(u_1-y_1^*)\lambda(y^*) - \tfrac{a}{q}|\le 2^{-j(2d-1)+\varepsilon_0 j}, \]
\[ |(u_2-y_1^*)\lambda(y^*) - \tfrac{a'}{q'}|\le 2^{-j(2d-1)+\varepsilon_0 j}. \]
Then,
\begin{equation}\label{eqn:tts1-partial}
|\lambda(y^*) - \tfrac{a^*}{q^*}| \le 2^{-j(2d-1)+\varepsilon_0 j +1},
\end{equation}
where $\tfrac{a^*}{q^*}=(u_2-u_1)^{-1}(\tfrac{a'}{q'}-\tfrac{a}{q})$ is a reduced rational with 
\begin{equation}\label{eqn:tts1-qsbd}
q^*\le q q' (u_2-u_1)\le 2^{3\varepsilon_0j + 7} d^2.
\end{equation}
With \eqref{eqn:tts1-partial} we have already obtained a somewhat decent rational approximation for $\lambda(y^*)$. However, to conclude $\lambda(y^*)\in X_j$, we need to show that the approximation is actually tighter by almost another factor of $2^{-j}$ on the right--hand side (see \eqref{eqn:pre-Xjdef}).
Denote the set of reduced rationals $\tfrac{a}{q}\in [0,1)$ with $q\le 2^{\varepsilon_0 j+1} d$ and $a\in [q]$ by $\mathscr{A}$. Then for each $\alpha\in\mathscr{A}$ we define
\[ \mathscr{F}_{\alpha} = \{ u\in\mathcal{E}\,:\,|(u-y_1^*)\lambda(y^*)-\alpha|_\T\le 2^{-(2d-1)j+\varepsilon_0 j} \}, \]
where $|\xi|_\T=\min_{z\in\Z} |\xi+z|\le |\xi|$.
By \eqref{eqn:tts1-claim1}, we have $\mathcal{E}\subset\cup_{\alpha\in\mathscr{A}} \mathscr{F}_\alpha$.
Since also $|\mathscr{A}|\le d^2 2^{2\varepsilon_0j+1}$, the pigeonhole principle and \eqref{eqn:tts1-asm} imply that there exists $\alpha_0=\tfrac{a_0}{q_0}\in\mathscr{A}$ such that 
\[|\mathscr{F}_{\alpha_0}|\ge 2^{j-3\varepsilon_0j-1} d^{-2}. \]

Now we invoke the pigeonhole principle again in the following form (this step can be skipped if $d>1$): for positive integers $N,k$ with $(2N+1)k^{-1}\ge 2$, cover a set
$A\subset [-N,N]\cap \Z$ with $k$ intervals, each of length $(2N+1)k^{-1}$. One of the intervals, call it $I$, must satisfy $|A\cap I|\ge |A|k^{-1}-1$.
Writing $v_1=\min\, A\cap I$ and $v_2=\max\, A\cap I$ we then have $|A|k^{-1}-2 \le v_2 - v_1 \le (2N+1)k^{-1}$.
Applying this fact to our situation with $N=2^{j+2}$, $A=\mathscr{F}_{\alpha_0}$, and $k=\lceil 2^{5\varepsilon_0 j}\rceil$, we exhibit $v_1, v_2\in\mathscr{F}_{\alpha_0}$ so that for $j\ge j_0$ large enough,
\begin{equation}\label{eqn:tts1-vprop}
2^{j-8\varepsilon_0 j-3} d^{-2} \le v_2-v_1 \le 2^{j-5\varepsilon_0 j+4}.
\end{equation}
By definition of $\mathscr{F}_{\alpha_0}$ there exist integers $\ell_1,\ell_2$ such that
\[ | (v_1-y_1^*)\lambda(y^*) - (\alpha_0+\ell_1)|\le 2^{-(2d-1)j+\varepsilon_0 j}, \]
\[ | (v_2-y_1^*)\lambda(y^*) - (\alpha_0+\ell_2)|\le 2^{-(2d-1)j+\varepsilon_0 j}. \]
This implies, using the lower bound in \eqref{eqn:tts1-vprop}, that
\begin{equation}\label{eqn:tts1-finalapprox}
| \lambda(y^*) - \tfrac{\ell_2-\ell_1}{v_2-v_1}| \le 2^{-2dj+9\varepsilon_0 j+3} d^{2}.
\end{equation}
We claim that
\begin{equation}\label{eqn:tts1-rateq}
\tfrac{\ell_2-\ell_1}{v_2-v_1} = \tfrac{a^*}{q^*}.
\end{equation}
Indeed, suppose not. Then, from \eqref{eqn:tts1-vprop} and \eqref{eqn:tts1-qsbd},
\[ |\tfrac{\ell_2-\ell_1}{v_2-v_1} - \tfrac{a^*}{q^*}| \ge \tfrac{1}{(v_2-v_1) q^*} \ge 2^{-j+2\varepsilon_0 j-11} d^{-2}. \]
On the other hand, from \eqref{eqn:tts1-partial} and \eqref{eqn:tts1-finalapprox},
\[ |\tfrac{\ell_2-\ell_1}{v_2-v_1} - \tfrac{a^*}{q^*}| \le 2^{-(2d-1)j+\varepsilon_0 j +2}, \]
for $j\ge j_0$ large enough. This yields a contradiction (again, for $j\ge j_0$ large enough). Thus, \eqref{eqn:tts1-rateq} holds. Summarizing, we have proved that
\[ |\lambda(y^*)-\tfrac{a^*}{q^*}| \le 2^{-2dj+10\varepsilon_0 j} \]
for $j\ge j_0$ large enough (from \eqref{eqn:tts1-rateq} and \eqref{eqn:tts1-finalapprox}). Further, $(a^*,q^*)=1$ and $q^*\le d^2 2^{3\varepsilon_0 j+7}\le 2^{\lfloor 10\varepsilon_0 j\rfloor}$ for large enough $j\ge j_0$. Recalling that we set $\varepsilon_0=\tfrac{1}{10} \varepsilon_1$, this means precisely that $\lambda(y^*)\in X_j$.

{\emph{Remarks.} 1. The argument simplifies slightly in the case $d>1$: in place of the upper bound in \eqref{eqn:tts1-vprop}, the trivial upper bound $2^{j+3}$ would be sufficient.\\
2. From the proof it is clear that the factor $\frac1{10}$ appearing in \eqref{eqn:tts1-Esmall} is not sharp. However, this is not relevant for our discussion.
}

\section{Minor arcs II: Proof of Proposition \ref{prop:minor}}\label{sec:err}

We will make use of the following fact.
\begin{lem}\label{lem:sobolev}
Let $\Lambda\subset\R$ be a disjoint union of intervals $(I_j)_{1\le j\le N}$ with $|I_j|\le \delta$, 
and $(m_\lambda)_{\lambda\in\Lambda}$ a family of bounded periodic functions on $\R^n$ such that
\begin{equation}\label{eqn:32assumption1}
\sup_{\lambda\in\Lambda} \|m_{\lambda}\|_{L^\infty(\T^n)} \le A,
\end{equation}
the function $I_j\to\C$, $\lambda\mapsto m_\lambda(\xi)$ is absolutely continuous for a.e. $\xi\in\R^n$ and every $j=1,\dots,N$, and
\begin{equation}\label{eqn:32assumption2}
\sup_{\lambda\in\Lambda} \|\partial_\lambda m_{\lambda}\|_{L^\infty(\T^n)} \le B,
\end{equation}
Then
\[ \| \sup_{\lambda\in\Lambda} |\fm{m_\lambda} f| \|_{\ell^2(\Z^n)} \le (N^{1/2} A + (2NAB\delta)^{1/2}) \|f\|_{\ell^2(\Z^n)}.  \] 
\end{lem}

The proof is via a standard argument using the fundamental theorem of calculus which we postpone to the end of this section. In order to apply Lemma \ref{lem:sobolev} to the multipliers $(E_{j,\lambda})_{\lambda\in X_j}$
we will prove that
\begin{equation}\label{eqn:err-Ejest}
|E_{j,\lambda}(\xi)| \lesssim 2^{-\gamma j}
\end{equation}
for some $\gamma>0$ only depending on $d,n$ (in particular, not depending on the choice of $\varepsilon_1$) and all $\lambda\in X_j, \xi\in\R^n$, $j\ge 1$.
Moreover, we have directly from the definitions \eqref{eqn:Ejdef}, \eqref{eqn:Ljsdef}, \eqref{eqn:Phijdef}, \eqref{eqn:mjdef} that for a.e. $\lambda\in X_j, \xi\in\R^n$ and every $j\ge 1$,
\begin{equation}\label{eqn:err-derivest}
|\partial_\lambda E_{j,\lambda}(\xi)| \lesssim 2^{2dj}.
\end{equation}
Then Lemma \ref{lem:sobolev} (with $\Lambda=X_j\cap [0,1)$, $m_\lambda=E_{j,\lambda}$, $N=|\mathfrak{A}_j|\le 2^{2\varepsilon_1 j}$, $\delta\le 2^{-2dj+\varepsilon_1 j+1}$) gives
\begin{equation}
\| \sup_{\lambda\in X_j} |\fm{E_{j,\lambda}} f| \|_{\ell^2(\Z^n)} \lesssim 2^{\frac12(3 \varepsilon_1 - \gamma) j} \|f\|_{\ell^2(\Z^n)}.
\end{equation}
Thus we obtain the claimed decay in $j$ as long as $\varepsilon_1 < \tfrac13 \gamma$. We turn our attention to proving \eqref{eqn:err-Ejest}. 

Assume $\lambda\in X_j$ (otherwise $E_{j,\lambda}(\xi)=0$). Fix $\varepsilon_2=2^{-5}$ (this can be replaced by any sufficiently small absolute constant with $\varepsilon_2 > \varepsilon_1$).
We define the \emph{major arcs}
\[ \mathfrak{M}_j = \bigcup_{\substack{(\alpha,\beta)\in\mathcal{R}_s,\\1\le s\le \varepsilon_2 j}} \mathfrak{M}_j(\alpha,\beta),\;\text{where} \]
\[ \mathfrak{M}_j(\alpha,\beta) = \{ (\lambda,\xi)\in\R\times\R^n\,:\,|\lambda-\alpha| \le 2^{-2dj+\varepsilon_2 j},\,|\xi-\beta|\le 2^{-j+\varepsilon_2 j} \}. \]

We need the following disjointness statement for the neighborhoods of the rationals involved in the sum defining $L_{j,\lambda}(\xi)$.

\begin{lem}\label{lem:Lj-disjoint}
For each $(\lambda,\xi)\in\R\times\R^n$ there exists at most one $(\alpha,\beta)$ with $(\alpha,\beta)\in\mathcal{R}_s$ for some $1\le s\le \varepsilon_2 j$ such that
\begin{equation}\label{eqn:ljs-disjointpf1}
S(\alpha,\beta) \Phi^*_{j,\lambda-\alpha}(\xi-\beta)\chi_s(\xi-\beta)\not =0.
\end{equation}
If that is the case and also $s\le \varepsilon_1 j$, then
\[ L_{j,\lambda}(\xi) = L^s_{j,\lambda}(\xi) = S(\alpha,\beta) \Phi^*_{j,\lambda-\alpha}(\xi-\beta)\chi_s(\xi-\beta).  \]
(Otherwise, $L_{j,\lambda}(\xi)=0$.)
\end{lem}

\begin{proof}
	Fix $(\lambda,\xi)\in \R\times\R^n$. Take $(\alpha,\beta)\in\mathcal{R}_s, (\alpha',\beta')\in\mathcal{R}_{s'}$ such that \eqref{eqn:ljs-disjointpf1} holds. Suppose that $\alpha\not=\alpha'$. Then
	\[ 2^{-2\varepsilon_2 j}\le 2^{-(s+s')}\le |\alpha-\alpha'|\le 2^{-2dj+\varepsilon_1 j + 1}. \]
	This is a contradiction. Thus, $\alpha=\alpha'$. Write $(\alpha,\beta)=(\tfrac{a}q,\tfrac{\mathbf{b}}q)$, $(\alpha',\beta')=(\tfrac{a'}{q'},\tfrac{\mathbf{b'}}{q'})$ with $(a,\mathbf{b},q)=(a',\mathbf{b}',q')=1$ and $2^{s-1}\le q<2^{s}$, $2^{s'-1}\le q'<2^{s'}$.  By Lemma \ref{lem:weylorth} and \eqref{eqn:ljs-disjointpf1} we have $(a,q)=1$ and $(a',q')=1$. But since $\alpha=\alpha'$, this implies $q=q'$ and thus $s=s'$. Taking another look at \eqref{eqn:ljs-disjointpf1} we see that $\beta=\beta'$ (by inspecting the support of $\chi_s=\chi_{s'}$). The claim about $L_{j,\lambda}(\xi)$ follows from the claim we just proved and \eqref{eqn:Ljdef}, \eqref{eqn:Ljsdef}.
\end{proof}

The proof of \eqref{eqn:err-Ejest} naturally splits into several cases.
\subsection*{\texorpdfstring{Case 1: $(\lambda,\xi)\in \mathfrak{M}_j$}{Case 1}}
Then there exist $1\le s_0\le \varepsilon_2 j$ and 
$(\alpha_0,\beta_0)\in\mathcal{R}_{s_0}$ such that 
$(\lambda,\xi)\in \mathfrak{M}_j(\alpha_0,\beta_0)$. 
From Lemma \ref{lem:approx} 
(with $\delta=2^{-j+\varepsilon_2 j}$, $q\le 2^{\varepsilon_2 j}$) we gather that
\begin{equation}\label{eqn:err-mjest1}
m_{j,\lambda}(\xi) = S(\alpha_0,\beta_0)\Phi_{j,\lambda-\alpha_0}(\xi-\beta_0) + O(2^{-j+2\varepsilon_2 j} )
\end{equation}
We distinguish two further cases.
\subsubsection*{\texorpdfstring{Case 1.1: $1\le s_0\le \varepsilon_1 j$}{Case 1.1}}
From Lemma \ref{lem:Lj-disjoint} we deduce
\[ L_{j,\lambda}(\xi) = L^{s_0}_{j,\lambda}(\xi) = S(\alpha_0,\beta_0) \Phi_{j,\lambda-\alpha_0}(\xi-\beta_0). \]
With \eqref{eqn:err-mjest1} this gives
\begin{align*}
|E_{j,\lambda}(\xi)| = |m_{j,\lambda}(\xi)-L_{j,\lambda}(\xi)| \lesssim 2^{-j+2\varepsilon_2 j}.
\end{align*}

\subsubsection*{\texorpdfstring{Case 1.2: $\varepsilon_1 j < s_0 \le \varepsilon_2 j$}{Case 1.2}}
We may write $\alpha_0=\tfrac{a_0}{q_0}$, $\beta_0 = \tfrac{\mathbf{b}_0}{q_0}$ with $(a_0,\mathbf{b}_0,q_0)=1$, $2^{s_0-1}\le q_0<2^{s_0}$. In particular, $q_0 \ge 2^{\lfloor \varepsilon_1 j\rfloor }$.

We claim that we must have $(a_0,q_0)>1$. Indeed, suppose $(a_0,q_0)=1$. Since $\lambda\in X_j$, there exists a reduced rational $\tfrac{a_1}{q_1}$ with $q_1<2^{\lfloor \varepsilon_1 j\rfloor}$ and
\[ |\tfrac{a_1}{q_1} - \lambda| \le 2^{-2dj+\varepsilon_1 j} \]
Since $q_0>q_1$, the reduced rationals $\tfrac{a_1}{q_1}$ and $\tfrac{a_0}{q_0}$ do not coincide. Therefore,
\[ 2^{-(\varepsilon_1+\varepsilon_2)j}\le \tfrac{1}{q_0 q_1}\le  |\tfrac{a_1}{q_1}-\tfrac{a_0}{q_0}| \le 2^{-2dj + \varepsilon_2 j + 1}. \]
This is a contradiction. Thus we must have $(a_0,q_0)>1$ and so $S(\alpha_0,\beta_0)=0$ by Lemma \ref{lem:weylorth}.
In particular, $|m_{j,\lambda}(\xi)|\lesssim 2^{-j+2\varepsilon_2 j}$ by \eqref{eqn:err-mjest1}.
Also, from Lemma \ref{lem:Lj-disjoint} we see that $L_{j,\lambda}(\xi)=0$.

\subsection*{\texorpdfstring{Case 2: $(\lambda,\xi)\not\in \mathfrak{M}_j$}{Case 2}}
In this case we bound
\[ |E_{j,\lambda}(\xi)| \le |m_{j,\lambda}(\xi)| + |L_{j,\lambda}(\xi)| \]
and estimate the two terms on the right--hand side separately.

Fix $\epsilon<\frac{\varepsilon_2}{n+1}$ and set $N=2^{j}$. By Dirichlet's approximation theorem there exist reduced fractions $\frac{a}{q}, \frac{b_1}{r_1}, \dots, \frac{b_n}{r_n}$ with $q\le N^{2d-\epsilon}$, $\max(r_1,\dots,r_n)\le N^{1-\epsilon}$ and
\[ |\lambda - \tfrac{a}{q}| \le \tfrac{1}{q} N^{-2d+\epsilon},\; |\xi_k-\tfrac{b_k}{r_k}|\le \tfrac{1}{r_k} N^{-1+\epsilon}\;\text{for}\;k=1,\dots,n. \] 
Setting $q_* = \mathrm{lcm}(q,r_1,\dots,r_n)$, we must have $q_*\ge 2^{\lfloor\varepsilon_2 j\rfloor}$ because $(\lambda,\xi)\not\in\mathfrak{M}_j$. Thus at least one of $q,r_1,\dots,r_n$ must be $\ge 2^{\epsilon j}$ (otherwise $q_*\le 2^{\epsilon (n+1) j}$ which is a contradiction because $\epsilon<\frac{\varepsilon_2}{n+1}$).
By Proposition \ref{prop:weylpower} we then obtain
\[ |m_{j,\lambda}(\xi)| \lesssim 2^{-\delta j}. \]
It remains to estimate $|L_{j,\lambda}(\xi)|$. Suppose that $L_{j,\lambda}(\xi)\not=0$. Then, by Lemma \ref{lem:Lj-disjoint} there exists $(\alpha,\beta)\in\mathcal{R}_s$ for some $1\le s\le \varepsilon_1 j$ such that
\begin{equation}\label{eqn:err-pf2}
L_{j,\lambda}(\xi)= S(\alpha,\beta)\Phi^*_{j,\lambda-\alpha}(\xi-\beta)\chi_s(\xi-\beta).
\end{equation}
Then $|\lambda-\alpha| \le 2^{-2dj+\varepsilon_2 j}.$ Since $(\lambda,\xi)\not\in\mathfrak{M}_j$,
\[ |\xi-\beta| \ge 2^{-j+\varepsilon_2 j}. \] 
With \eqref{eqn:Phij-est} and \eqref{eqn:err-pf2}, this implies
\[ |L_{j,\lambda}(\xi)| \le |\Phi_{j,\lambda-\alpha}(\xi-\beta)| \lesssim 2^{-\frac{\varepsilon_2}{2d} j}. \]

\begin{proof}[Proof of Lemma \ref{lem:sobolev}]
By the fundamental theorem of calculus, we have for absolutely continuous $g:[a,b]\to \C$,
\[ \sup_{\lambda \in [a,b]} |g(\lambda)|^2 \le |g(a)|^2 + 2 \int_a^b |g(t)| |g'(t)| dt. \]
Hence,
\begin{equation}\label{eqn:sobolev-pf1}
\| \sup_{\lambda \in \Lambda } |\fm{m_\lambda}f| \|_{\ell^2(\Z^n)}^2 \le \sum_{j=1}^N \|\fm{m_{\inf\,I_j}} f\|^2_{\ell^2(\Z^n)} 
\end{equation}
\[\hspace{4cm} +\quad2 \sum_{j=1}^N\sum_{x\in\Z^n} \int_{I_j} |\fm{m_t} f(x)| |\partial_\lambda \fm{m_t} f(x)| dt.  \]
By the Cauchy--Schwarz inequality and Fubini's theorem,
\[\sum_{x\in\Z^n} \int_{I_j} |\fm{m_t} f(x)| |\partial_\lambda \fm{m_t} f(x)| dt \]
\[ \le
\Big(\int_{I_j} \|\fm{m_t} f\|_{\ell^2(\Z^n)}^2 dt \Big)^{1/2} \Big( \int_{I_j} \|\fm{\partial_\lambda m_t} f\|_{\ell^2(\Z^n)}^2 dt \Big)^{1/2}. \]
Combining this with \eqref{eqn:sobolev-pf1} and using Plancherel's theorem with the assumptions \eqref{eqn:32assumption1}, \eqref{eqn:32assumption2} we obtain the claim.\\
{\emph{Remark.} Observe that the same argument works for $\ell^p$ with $p\not=2$ and more general families of operators.}
\end{proof}

\section{Major arcs I: Proof of Proposition \ref{prop:weylsums}}\label{sec:wtt}
Note that since $\mathscr{L}_{s,\alpha}[1] = \mathscr{L}_{s,\alpha+1}[1]$, we may restrict the supremum to $\alpha\in\mathcal{A}_s\cap [0,1)$, without loss of generality. 
Also, from \eqref{eqn:Lsmkernel} we have
\[ \mathcal{F}^{-1}_{\Z^n}[\mathscr{L}_{s,\alpha}[1]](y) = 
\sum_{\beta\in \mathcal{B}_s(\alpha)\cap [0,1)^n} S(\alpha,\beta) e(\beta\cdot y) \phi_s(y), \]
where $\phi_s = \mathcal{F}_{\R^n}^{-1} [ \chi_s ]$. Note that $\|\phi_s\|_{L^1(\R^n)}\approx 1$.
For an arbitrary function $\alpha:\Z^n\to \mathcal{A}_s\cap [0,1)$ we define  
\begin{equation}\label{eqn:wtt-kdef}
\mathfrak{K}_{s,\alpha}(x,y) = \mathcal{F}_{\Z^n}^{-1}[\mathscr{L}_{s,\alpha(x)}[1]](x-y).
\end{equation}
Then Proposition \ref{prop:weylsums} is a consequence of the following.
\begin{prop}\label{prop:wtt-main}
There exists $\gamma>0$ depending only on $d, n$ such that
\[ \|T_{\mathfrak{K}_{s,\alpha}} f\|_{\ell^2(\Z^n)} \lesssim 2^{-\gamma s} \|f\|_{\ell^2(\Z^n)}, \]
with the implicit constant only depending on $d$ and $n$, but not on the functions $\alpha, f$. (The notation $T_{\mathfrak{K}_{s,\alpha}}$ is defined in \eqref{eqn:pre-Tkdef}.)
\end{prop}
{\emph{Remark.} The proof shows that the same result holds with $\ell^2$ replaced by $\ell^p$ for every $p\in (1,\infty)$ (with decay rate depending on $p$).\\
}

For every $x\in\Z^n$ there exist $\mathrm{q}(x)\in \Z\cap [2^{s-1},2^s)$ and $\mathrm{a}(x)\in [\mathrm{q}(x)]$ with $(\mathrm{a}(x), \mathrm{q}(x))=1$ such that
\[ \alpha(x) = \tfrac{\mathrm{a}(x)}{\mathrm{q}(x)}. \]
For the proof we will employ a $TT^*$--argument. We begin by computing the kernel of $TT^*$. Note that $T_{\mathfrak{K}_{s,\alpha}} T_{\mathfrak{K}_{s,\alpha}}^*  = T_{\mathfrak{K}^\sharp_{s,\alpha}},$
where
\[ \mathfrak{K}^\sharp_{s,\alpha} (x,y) = \sum_{z\in\Z^n} \mathfrak{K}_{s,\alpha}(x,z) \overline{\mathfrak{K}_{s,\alpha}(y,z)}. \]
From \eqref{eqn:wtt-kdef},
\[ \mathfrak{K}^\sharp_{s,\alpha}(x,y) = \sum_{\substack{\beta\in\mathcal{B}_s(\alpha(x))\cap [0,1)^n,\\\beta'\in\mathcal{B}_s(\alpha(y))\cap [0,1)^n }}  S(\alpha(x),\beta) e(x\cdot \beta) \overline{S(\alpha(y),\beta')} e(-y\cdot \beta')  \]
\[\hspace{2cm} \times \left[ \sum_{z\in\Z^n} \phi_s(x-z) \overline{\phi_s(y-z)}  e(z\cdot (\beta'-\beta)) \right] \]
Next we claim that for every $\beta,\beta'\in \mathcal{B}_s^\sharp\cap [0,1)^n$ with $\beta\not=\beta'$ it holds that
\begin{equation}\label{eqn:wtt-poisson1}
\sum_{z\in\Z^n} \phi_s(x-z) \overline{\phi_s(y-z)}  e(z\cdot (\beta'-\beta)) = 0.
\end{equation}
To see this, define a Schwartz function on $\R^n$ by
\[ \Xi(t) = \phi_s(x-t) \overline{\phi_s(y-t)} e(t (\beta'-\beta)),\;\quad(t\in\R^n). \]
Then
\[ \widehat{\Xi}(\xi) = [\mathrm{M}_{x} \widetilde{\chi}_s* \mathrm{M}_{-y} \widetilde{\chi}_s ] (\xi+\beta-\beta'), \]
where we used the notation $\mathrm{M}_u g(x) = e(u\cdot x) g(x)$. From the definitions of $\widetilde{\chi}_s$ and $\mathcal{B}^\sharp_s$ we then have for $\xi\in\Z^n$ that $\widehat{\Xi}(\xi)=0$ unless $\xi+\beta-\beta'=0$. However, $\beta,\beta' \in [0,1)^n$ and $\beta\not=\beta'$ imply $\beta-\beta'\not\in\Z^n$.
Hence, by the Poisson summation formula the left--hand side of \eqref{eqn:wtt-poisson1} is equal to \[ \sum_{z\in\Z^n} \Xi(z) = \sum_{\xi\in\Z^n} \widehat{\Xi}(\xi) = 0. \]
As a consequence,
\begin{equation}\label{eqn:wtt-ker0}
\mathfrak{K}^\sharp_{s,\alpha}(x,y) = \overline{\kappa_{s,\alpha}(x,y)}\cdot  [\phi_s*\overline{\phi_s}](x-y),
\end{equation}
where we set
\begin{equation}\label{eqn:wtt-pf1}
\kappa_{s,\alpha}(x,y) = \sum_{\beta\in \mathcal{B}_s(\alpha(x))\cap \mathcal{B}_s(\alpha(y))\cap [0,1)^n} S(\alpha(y),\beta) \overline{S(\alpha(x),\beta)} e((y-x)\cdot\beta).
\end{equation}
For the following computation we fix $(x,y)\in \Z^n\times \Z^n$ and write
\[ a =  \mathrm{a}(y),\quad q = \mathrm{q}(y),\quad a' = \mathrm{a}(x),\quad q' = \mathrm{q}(x) \]
for short. As a consequence of Lemma \ref{lem:weylorth}, we may assume $(a,q)=(a',q')=1$ and read the sum over $\beta$ in \eqref{eqn:wtt-pf1} as running over the set
\[ \{ \tfrac{\mathbf{b}}{q}\;:\;\mathbf{b}\in [q]^n \}\cap \{ \tfrac{\mathbf{b}}{q'}\;:\;\mathbf{b}\in [q']^n \}, \]
which is equal to
\[ \{ \tfrac{\mathbf{b}}{q_\flat}\,:\,\mathbf{b}\in [q_\flat]^n \}, \]
where we have set $q_\flat=(q,q')$. Thus,
\[ \kappa_{s,\alpha}(x,y) = \sum_{\mathbf{b}\in [q_\flat]^n} S(\tfrac{a }{q},\tfrac{\mathbf{b}}{q_\flat}) \overline{S(\tfrac{a'}{q'},\tfrac{\mathbf{b}}{q_\flat})} e((y-x)\cdot \tfrac{\mathbf{b}}{q_\flat}). \]
Expanding the exponential sums by \eqref{eqn:defcompleteweylsum}, we can rewrite this as
\[ (qq')^{-n} \sum_{r\in [q]^n, r'\in [q']^n} e(\tfrac{a}{q} |r|^{2d} - \tfrac{a'}{q'} |r'|^{2d} )
\left[ \sum_{\mathbf{b}\in [q_\flat]^n} e(\tfrac{\mathbf{b}}{q_\flat}\cdot (r-r'+y-x))\right],  \]
which, in view of the relation $N^{-1}\sum_{l\in [N]} e(\tfrac{l z}N) = \mathbf{1}_{z\equiv 0\;(\text{mod}\;N)}$, is equal to
\begin{equation}\label{eqn:wtt-ttsker}
(\tfrac{q'}{q_\flat})^{-n} \sum_{u\in [\tfrac{q'}{q_\flat}]^n} q^{-n} \sum_{r\in [q]^n} e(\tfrac{a}{q}|r|^{2d} - \tfrac{a'}{q'}|r+y-x+u\cdot q_\flat|^{2d}).
\end{equation}
Inspection of this exponential sum reveals several scenarios in which no cancellation can be expected. For instance, a typical case where \eqref{eqn:wtt-ttsker} exhibits no cancellation is when $a=a'$, $q=q'$ and $y-x$ is divisible by $q$ (then $\kappa_{s,\alpha}(x,y)=1$).
Additional degeneracies arise in the case $d=1$, requiring a more careful analysis. For $w\in \Z^n$ we define
\begin{equation}\label{eqn:wtt-case2-1}
S_{x,y}(w) = q^{-n} \sum_{r\in[q]^n} e(\tfrac{a}{q}|r|^{2d} - \tfrac{a'}{q'}|r+w|^{2d}).
\end{equation}
In the case $d\ge 2$ it will suffice to exploit cancellation from the exponential sum \eqref{eqn:wtt-case2-1}, whereas in the case $d=1$ we will sometimes need to make use of cancellation from the sum over $u$ in \eqref{eqn:wtt-ttsker}.

\subsection*{\texorpdfstring{The case $d\ge 2$}{The case d > 1}}
Viewing the phase in \eqref{eqn:wtt-case2-1} as a polynomial in $r$, the coefficient of $r_1^{2d-1}$ is equal to $\tfrac{-2da' w_1}{q'}.$
This leads us to define
\[ \mathcal{E}_x = \{  w\in\Z^n\,:\, (2dw_1,\mathrm{q}(x))\ge 2^{s/2} \}. \]
By sorting modulo $\mathrm{q}(x)$ and counting divisors of $\mathrm{q}(x)$ we see that for $z\in\Z^n, N\ge 2^s$ and every $\varepsilon>0$,
\begin{equation}\label{eqn:wtt-main-smallset}
N^{-n} |\mathcal{E}_x \cap (z+[N]^n)| \lesssim_\varepsilon 2^{-s/2+\varepsilon s}.
\end{equation}
If $w\not\in\mathcal{E}_x$, then Proposition \ref{prop:weylpower} (with $R=q\le R^{(2d-1)-\varepsilon}$, crucially using $d\ge 2$) yields
\[ |S_{x,y}(w)| \lesssim 2^{-\gamma s} \]
for some sufficiently small $\gamma\in (0,\frac12)$ depending on $d$ and $n$.
Using the triangle inequality on the sum over $u$ in \eqref{eqn:wtt-ttsker} leads to the estimation
\begin{equation}\label{eqn:wtt-main-est}
|\kappa_{s,\alpha}(x,y)| \lesssim  2^{-\gamma s} + \sum_{\nu|\mathrm{q}(x)} (\tfrac{\mathrm{q}(x)}{\nu})^{-n}\sum_{u\in [\mathrm{q}(x)/\nu]^n} \mathbf{1}_{\mathcal{E}_x}(y-x+u\cdot \nu) ,
\end{equation}
where we have removed the $(x,y)$--dependence of $q_\flat=(\mathrm{q}(x),\mathrm{q}(y))$ by summing over all divisors of $\mathrm{q}(x)$. Hence, recalling \eqref{eqn:wtt-ker0}, we see for every $x\in\Z^n$ that
\[ \sum_{y\in\Z^n} |\mathfrak{K}^\sharp_{s,\alpha}(x,y)| \lesssim 2^{-\gamma s} + \tau(q(x)) \sup_{u\in\Z^n} \sum_{y\in \Z^n} \mathbf{1}_{\mathcal{E}_x}(y-x+u) |\phi_s*\phi_s|(x-y), \]
where $\tau(q)$ denotes the number of divisors of $q$. Using the standard divisor bound $\tau(q)\lesssim_\varepsilon q^{\varepsilon}$, \eqref{eqn:wtt-main-smallset} and rapid decay of $\phi_s*\phi_s$, we obtain 
\[ \sum_{y\in\Z^n} |\mathfrak{K}^\sharp_{s,\alpha}(x,y)|\lesssim 2^{-\gamma s} \]
for every $x\in\Z^n$. Since also $\mathfrak{K}^\sharp_{s,\alpha}(x,y) = \overline{\mathfrak{K}^\sharp_{s,\alpha}(y,x)}$, we infer from Schur's test that
\[ \|T_{\mathfrak{K}^\sharp_{s,\alpha}}\|_{\ell^2(\Z^n)\to \ell^2(\Z^n)} \lesssim 2^{-\gamma s}. \]
This concludes the proof of Proposition \ref{prop:wtt-main}.

\subsection*{\texorpdfstring{The case $d=1$}{The case d=1}}
First assume that $q_\flat = (\mathrm{q}(x), \mathrm{q}(y)) \le 2^{s/3}.$ 
Then $\frac{q'}{q_\flat}\ge 2^{2s/3-1}$. 
Viewing the phase in \eqref{eqn:wtt-ttsker} as a polynomial in $u$, 
the coefficient of $u_1^2$ is $-\frac{a' q_\flat^2}{q'}$ which equals a reduced rational 
with denominator in $[2^{s/3}, 2^s]\cap\Z$. 
Thus, applying Proposition \ref{prop:weylpower} to the exponential sum over $u$ yields
\[ |\kappa_{s,\alpha}(x,y)| \lesssim 2^{-\gamma s} \]
for a small enough $\gamma>0$. Next we handle the case that $q_\flat \ge 2^{s/3}$. We will exploit cancellation from the summation over $r$ in \eqref{eqn:wtt-ttsker}. The exponential sum on the right--hand side of \eqref{eqn:wtt-case2-1} factors into $n$ one--dimensional sums. It will be enough to estimate the first factor, which is given by
\[ I = q^{-1} \sum_{r_1\in [q]} e( \tfrac{A}Q r_1^2 - \tfrac{2a' w_1}{q'} r_1 ), \]
where $\tfrac{A}Q = \tfrac{a}q-\tfrac{a'}{q'}$ with $(A,Q)=1$.
We are led to distinguish two cases. Suppose that $Q\ge 2^{s/3}$. 
Then, since also $Q\le \frac{qq'}{q_\flat}\le 2^{5s/3}$, we 
may apply Proposition \ref{prop:weylpower} to obtain
\begin{equation}\label{eqn:wtt-d1-1}
|I| \lesssim 2^{-\delta s}
\end{equation}
for some small enough $\delta>0$. On the other hand, assume $Q\le 2^{s/3}$. Then, by reorganizing the summation modulo $Q$,
\[ I = q^{-1} \big[\sum_{s \in [Q]} e(\tfrac{A}Q s^2 - \tfrac{2a' w_1}{q'} s)\big]\cdot \big[ \sum_{u\in [ M ]} e(-\tfrac{2a' w_1 Q}{q'} u) \big] + O(2^{-\frac23 s}),  \]
where $M=\lfloor \tfrac{q}{Q}\rfloor$. Summing the geometric sum over $u$ and using the triangle inequality on the sum over $s$ we get
\begin{equation}\label{eqn:wtt-d1-2}
|I| \lesssim 2^{-\frac23 s} |1-e(\tfrac{2a' w_1 Q}{q'})|^{-1}\lesssim 2^{-\frac23 s} |\tfrac{2a' w_1 Q}{q'} |^{-1}_\T,
\end{equation}
where $|\xi|_\T = \min_{z\in\Z} |\xi+z|$. Note that $Q$ depends on both $x$ and $y$. To remove the dependence on $y$ we define for a positive integer $\upsilon\le 2^{s/3}$ the set
\begin{equation}\label{eqn:wtt-d1-3}
E^{(\upsilon)}_x = \{ w_1 \in \Z \,:\, |\tfrac{2a'w_1 \upsilon}{q'}|_\T \le 2^{-s/2}  \}.
\end{equation}
Let $\iota = (2\upsilon,q')\lesssim 2^{s/3}$ and $\mathfrak{q}=\tfrac{q'}{\iota}$. Let $\mathcal{R}\subset\Z$ be a complete residue system modulo $\mathfrak{q}$. Then $a'\tfrac{2\upsilon}{\iota}\mathcal{R}$ is also a complete residue system modulo $\mathfrak{q}$. Thus 
\[|E_x^{(\upsilon)}\cap \mathcal{R}|=|\{ \ell\in [\mathfrak{q}]\,:\,|\ell/\mathfrak{q}|_\T \le 2^{-s/2} \}| \lesssim \mathfrak{q} 2^{-s/2}.\]
Since $\mathfrak{q}\le q' < 2^s$, we then have for every $N\ge 2^s$ and $z\in\Z$,
\[ N^{-1} |E_x^{(\upsilon)}\cap (z+[N])|\lesssim 2^{-s/2}. \]
Define
\begin{equation}\label{eqn:wtt-cased1-set}
\mathcal{E}_x = \bigcup_{\upsilon\le 2^{s/3}} \{ w\in\Z^n\,:\,w_1\in E_x^{(\upsilon)} \}. 
\end{equation}
Then if $w\not\in \mathcal{E}_x$, we gather from \eqref{eqn:wtt-case2-1}, \eqref{eqn:wtt-d1-1}, \eqref{eqn:wtt-d1-2}, \eqref{eqn:wtt-d1-3} that
\[ |S_{x,y}(w)| \le |I| \lesssim \max(2^{-\frac16 s}, 2^{-\delta s}) \]
and for every $z\in\Z^n$ and $N\ge 2^s$,
\[ N^{-n} |\mathcal{E}_x \cap (z + [N]^n)| \lesssim 2^{-\frac16 s}. \]

The fact that we have chosen the exceptional set $\mathcal{E}_x$ only depending on $x$ (as opposed to both $x$ and $y$) allows us to recycle the crude argument using Schur's test seen in the case $d\ge 2$. Indeed, summarizing the above we have shown that \eqref{eqn:wtt-main-est} again holds for all $(x,y)\in \Z^n\times\Z^n$ with $\mathcal{E}_x$ defined as in \eqref{eqn:wtt-cased1-set} (and $\gamma>0$ small enough, possibly different from above). This completes the proof of Proposition \ref{prop:wtt-main}.

\section{Major arcs II: Proof of Proposition \ref{prop:multifreq}}\label{sec:mft}
Before we begin with the proof we collect some preliminary results.
First note from \eqref{eqn:Lsmkernel} that for every bounded function $m$ on $\R^n$,
\begin{equation}\label{eqn:Lsmspace}
\fm{\mathscr{L}_{s,\alpha}[m]}f(x) = \sum_{\beta \in \mathcal{B}_s(\alpha)\cap [0,1)^n} S(\alpha,\beta) e(x\cdot\beta) \mathcal{F}_{\R^n}^{-1}[ m\, \chi_s] * M_{-\beta} f (x),
\end{equation}
where $M_{a} f(x) = e(x\cdot a) f(x)$ denotes modulation by $a$ (and $*$ must denote convolution on $\Z^n$, since $f$ is only defined on $\Z^n$).
The factorization \eqref{eqn:Lfact} will allow us to prove the following consequence of Proposition \ref{prop:weylsums}.
\begin{lem}\label{lem:basic}
Let $\mathcal{I}$ be a countable set and $(m_{\nu})_{\nu\in\mathcal{I}}$ a family of bounded functions on $\R^n$. Then there exists $\gamma>0$ such that
for every $s\ge 1$,
\[ \| \sup_{\alpha\in \mathcal{A}_s} | \fm{\mathscr{L}_{s,\alpha}[m_\nu]}f(x) | \|_{\ell_{x,\nu}^2(\Z^n\times \mathcal{I})} \lesssim_{d,n} 2^{-\gamma s} \Big( \sup_{|\xi|\le 1} \|m_\nu(\xi)\|_{\ell^2_\nu(\mathcal{I})} \Big) \|f\|_{\ell^2(\Z^n)}. \]
\end{lem}

Observe that Proposition \ref{prop:weylsums} is the special case $\mathcal{I}=\{\cdot\}$, $m_\cdot\equiv 1$.
We also need the following result which can be seen as a variant of Lemma \ref{lem:basic} (in the case $\mathcal{I}=\{\cdot\}$) for maximally truncated singular integrals.

\begin{lem}\label{lem:mf-truncsingint}
For $j\ge 1$ let $\mathcal{K}_j$ be a mean zero $C^1$ function supported on $\{ |x|\asymp 2^{j} \}$ so that there exists a constant $A>0$ with
\begin{equation}\label{eqn:CZcond}
2^{jn} |\mathcal{K}_j(x)| + 2^{j(n+1)} |\nabla \mathcal{K}_j(x)| \le A
\end{equation}
for all $j\ge 1$ and $x\in\R^n$. 
Write
\[ \mathcal{K}^{a,b} (x) = \sum_{a\le j<b} \mathcal{K}_j (x). \]
Then there exists $\gamma>0$ so that for all $s\ge 1$,
\begin{equation}\label{eqn:truncgoal}
\| \sup_{J\ge 1} \sup_{\alpha\in\mathcal{A}_s} |\fm{\mathscr{L}_{s,\alpha}[\widehat{\mathcal{K}^{0,J}}]}f |\|_{\ell^2(\Z^n)} \lesssim 2^{-\gamma s} \|f\|_{\ell^2(\Z^n)}.
\end{equation}
Here the implicit constant depends only on $A,d,n$ and $\widehat{\cdot}$ (necessarily) denotes the Fourier transform on $\R^n$.
\end{lem}

The proofs of Lemma \ref{lem:basic} and Lemma \ref{lem:mf-truncsingint} are postponed to \S \ref{sec:lemma1pf} and \S \ref{sec:lemma2pf}, respectively.
We now begin with the proof of Proposition \ref{prop:multifreq}.
The maximal operator in question can be bounded by 
\begin{equation}\label{eqn:mftpf-1}
	\sup_{\alpha\in\mathcal{A}_s} \sup_{|\mu|\le 2^{-C_0 s}} \Big| \sum_{j\ge \varepsilon_1^{-1} s} 
	\fm{\mathscr{L}_{s,\alpha}[ \Phi_{j,\mu} ]}f(x)  \Big|,
\end{equation}
where $C_0=(2d-1)\varepsilon_1^{-1}$ is a large constant. 

Motivated by phase considerations, we introduce a frequency scale parameter $\ell\in\Z$ to partition the scales $j$ for every fixed $\mu$ as
\[  \mathcal{J}_{\ell,\mu} = \{ j\ge \varepsilon_1^{-1}s\,:\,|\mu| 2^{2dj} \asymp 2^{\ell}  \}, \]
where $A\asymp B$ means $\tfrac12 B\le A\le 2B$. Note that $\mathcal{J}_{\ell,\mu}$ has at most one element and is often empty. Define
\[ \widetilde{\Phi}_{\ell,\mu} (\xi) = \sum_{j\in\mathcal{J}_{\ell,\mu}} \Phi_{j,\mu}(\xi). \]
Let $C_1$ be a large positive constants that is to be determined later. We distinguish three cases
\[ \mathcal{L}_1 = \{ \ell\in\Z\,:\,\ell \ge C_1 s \}, \]
\[ \mathcal{L}_2 = \{ \ell\in\Z\,:\, -C_1 s < \ell < C_1 s \}, \]
\[ \mathcal{L}_3 = \{ \ell\in\Z\,:\, \ell\le -C_1 s \} \]
and bound the $\ell^2$ norm of \eqref{eqn:mftpf-1} accordingly by
\begin{equation}\label{eqn:mftpf-2}
	\sum_{i=1}^3 \Big\| \sup_{\alpha\in\mathcal{A}_s} \sup_{|\mu|\le 2^{-C_0 s}} \Big| \sum_{\ell\in\mathcal{L}_i} 
	\fm{\mathscr{L}_{s,\alpha}[ \widetilde{\Phi}_{\ell,\mu} ]}f  \Big|  \Big\|_{\ell^2(\Z^n)}
\end{equation}
It remains to bound these three summands separately, which is the content of the next three subsections.

\subsection{The high frequency case\texorpdfstring{: $\ell\in \mathcal{L}_1$}{}}
This is the easiest case. Here $\ell\ge C_1 s$. From Stein and Wainger's theorem \cite{SW01} we see
\[ \|\sup_{\mu\in\R} |\fm{\widetilde{\Phi}_{\ell,\mu}} g| \|_{L^2(\R^n)} \lesssim 2^{-\gamma \ell} \|g\|_{L^2(\R^n)}  \]
for some $\gamma>0$. Since $\# \mathcal{R}_s\lesssim 2^{Cs}$ 
for some $C>0$, we may choose $C_1$ large enough so that the trivial bound $|S(\alpha,\beta)|\le 1$, the triangle inequality on the summation over $\beta$
and a transference argument using Lemma \ref{lem:pre-transf}
yield that the first summand in \eqref{eqn:mftpf-2}
is $\lesssim 2^{-\gamma' s} \|f\|_{\ell^2(\Z^n)}$ for some $\gamma'>0$.

\subsection{The intermediate frequency case\texorpdfstring{: $\ell\in \mathcal{L}_2$}{}}\label{sec:medfreq}
First note that $\# \mathcal{L}_2\lesssim s$, so we may consider the terms for each fixed $\ell\in\mathcal{L}_2$ separately. 
The arguments differ slightly depending on the sign of $\ell$ with the terms for $\ell\ge 0$ being the more problematic ones.
Set $\ell^+=\max(0,\ell)$. 
We begin estimating
\[ \| \sup_{\alpha\in\mathcal{A}_s} \sup_{|\mu|\le 2^{-C_0 s}} 
|\fm{\mathscr{L}_{s,\alpha}[ \widetilde{\Phi}_{\ell,\mu} ]}f|
\|_{\ell^2(\Z^n)} \]
by
\begin{equation}\label{eqn:mftpf-3}
\| \sup_{j\ge \varepsilon_1^{-1} s} \sup_{\alpha\in\mathcal{A}_s} \sup_{|\mu|\asymp 2^{\ell-2dj}} 
|\fm{\mathscr{L}_{s,\alpha}[ \Phi_{j,\mu} ]}f|
\|_{\ell^2(\Z^n)}.
\end{equation}
Let $\vartheta$ denote an appropriately compactly supported, smooth and non-negative function so that $\int\vartheta=1$ and $\vartheta_{j}(x)=2^{-jn}\vartheta(2^{-j}x)$.
With $\mu$ so that $|\mu|\asymp 2^{\ell-2dj}$ we define
\[ \Phi^\flat_{j,\mu}(\xi)  = \Phi_{j,\mu}(\xi)  - \Phi_{j,\mu}(0) \widehat{\vartheta_{j-\ell^+}}(\xi).\]
The definition is made so that $\Phi^\flat_{j,\mu}(0)=0$ and $\Phi^\flat_{j,\mu}$ satisfies favorable estimates that will be stated below
(also see \cite[Lemma 3.22]{KL17}).
Observe that $|\Phi_{j,\mu}(0)|\lesssim 1$ (there is also decay in $|\ell|$ but we will only need that later on).
We now estimate \eqref{eqn:mftpf-3} by the sum of 
\begin{equation}\label{eqn:mftpf-3.1}
	\Big( \sum_{j\ge \varepsilon_1^{-1} s} \| \sup_{\alpha\in\mathcal{A}_s} \sup_{|\mu|\asymp 2^{\ell-2dj}} 
	|\fm{\mathscr{L}_{s,\alpha}[ \Phi^\flat_{j,\mu} ]}f|
	\|^2_{\ell^2(\Z^n)} \Big)^{1/2}
  \end{equation}
and 
\begin{equation}\label{eqn:mftpf-3.2}
	\| \sup_{j\ge \varepsilon_1^{-1} s} \sup_{\alpha\in\mathcal{A}_s} 
	|\fm{\mathscr{L}_{s,\alpha}[ \widehat{\vartheta_{j-\ell^+}} ]}f|
	\|_{\ell^2(\Z^n)}. 
\end{equation}
We begin with handling the latter term. Since $\ell^+\le C_1 s$ we can bound \eqref{eqn:mftpf-3.2} by
\begin{equation}\label{eqn:mftpf-3.3}
	\| \sup_{J\ge 1} \sup_{\alpha\in\mathcal{A}_s} 
	|\fm{\mathscr{L}_{s,\alpha}[ \widehat{\vartheta_{J}} ]}f|
	\|_{\ell^2(\Z^n)} 
\end{equation}
as long as $\varepsilon_1$ is chosen small enough so that $\varepsilon_1^{-1}\ge C_1$ (which is possible, since the choice of $C_1$ is independent of that of $\varepsilon_1$).
Expanding into a telescoping sum
\[ \vartheta_{J} = \vartheta_0 + \sum_{0\le j< J} (\vartheta_{j+1} - \vartheta_{j}), \]
we bound \eqref{eqn:mftpf-3.3} by
\begin{equation}\label{eqn:mftpf-3.4}
	\| \sup_{\alpha\in\mathcal{A}_s} 
	|\fm{\mathscr{L}_{s,\alpha}[ \widehat{\vartheta_0} ]}f|
	\|_{\ell^2(\Z^n)} + \| \sup_{J\ge 1} \sup_{\alpha\in\mathcal{A}_s} 
	|\fm{\mathscr{L}_{s,\alpha}[ \sum_{0\le j< J} (\widehat{\vartheta_{j+1}} - \widehat{\vartheta_{j}}) ]}f|
	\|_{\ell^2(\Z^n)}.
\end{equation}
Using Lemma \ref{lem:basic} for the first term and Lemma \ref{lem:mf-truncsingint} for the second term we see that the previous display is $\lesssim 2^{-\gamma s} \|f\|_{\ell^2(\Z^n)}$ as required.

It remains to estimate the main term \eqref{eqn:mftpf-3.1}.
The uncertainty principle suggests that the value of $|\Phi^\flat_{j,\mu}(D) g(x)|$ stays approximately constant as $\mu$ varies over an interval of length $\lesssim 2^{-2dj}$.
This motivates the following standard argument (similar to \S \ref{sec:err}). Define $I=\{ |\mu|\asymp 2^{\ell-2dj} \}=[-b,-a]\cup [a,b]$ with $0<a<b$, set $\delta=2^{-2dj}$ 
 and define 
\[ \mathfrak{F} = \mathfrak{F}_{\ell,j} = \bigcup_{0\le k<\lceil\tfrac{b-a}{\delta}\rceil} \{ a + k \delta, -b + k\delta \}. \]
The set $\mathfrak{F}_{\ell,j}$ has cardinality $\lceil 3\cdot 2^{\ell}\rceil\approx 2^{\ell^+}$.
For every differentiable function $G:I\to \C$ we have by the fundamental theorem of calculus,
\[\sup_{\mu\in I} |G(\mu)| \le \|G\|_{\ell^2(\mathfrak{F})} + \delta \int_0^1 \|G'(\mu+t\delta)\|_{\ell^2_\mu(\mathfrak{F})}\, dt. \]
Using this with
\[  G(\mu) = \fm{\mathscr{L}_{s,\alpha}[ \Phi^\flat_{j,\mu} ]}f(x) \]
we bound \eqref{eqn:mftpf-3.1} by
\[ \Big(  \sum_{j\ge \varepsilon_1^{-1} s} \sum_{\mu\in \mathfrak{F}_{\ell,j}} \| \sup_{\alpha\in\mathcal{A}_s}  
|\fm{\mathscr{L}_{s,\alpha}[ \Phi^\flat_{j,\mu} ]}f|
\|^2_{\ell^2(\Z^n)} \Big)^{1/2} \]
\[+ \int_0^{1} \Big( \sum_{j\ge \varepsilon_1^{-1} s} \sum_{\mu\in \mathfrak{F}_{\ell,j}} 2^{-4dj}\| \sup_{\alpha\in\mathcal{A}_s}  
|\fm{\mathscr{L}_{s,\alpha}[ \partial_\mu \Phi^\flat_{j,\mu+2^{-2dj}t} ]}f|
\|^2_{\ell^2(\Z^n)} \Big)^{1/2} dt.\]
Applying Lemma \ref{lem:basic}, we see that \eqref{eqn:mftpf-3.1} is  $\le A\, 2^{-\gamma s} \|f\|_{\ell^2(\Z^n)}$
with 
\[ A^2 \lesssim \sup_{|\xi|\le 1} \sum_{j\ge \varepsilon_1^{-1} s} \sum_{\mu\in\mathfrak{F}_{\ell,j}} |\Phi^\flat_{j,\mu}(\xi)|^2 + \sup_{|\xi|\le 1}\sup_{t\in [0,1]} \sum_{j\ge \varepsilon_1^{-1} s} \sum_{\mu\in\mathfrak{F}_{\ell,j}} 2^{-4dj} |\partial_\mu \Phi^\flat_{j,\mu+2^{-2dj}t}(\xi)|^2. \]
To show that $A\lesssim 1$ we will use the following estimates.
Let $j\ge 1$ and $|\mu|\asymp 2^{\ell-2dj}$.
If $\ell\ge 0$, then for all $N\ge 0$,
\begin{equation}\label{eqn:mftpf-4.0} 
|\Phi^\flat_{j,\mu}(\xi)| + 2^{-2dj} |\partial_\mu \Phi^\flat_{j,\mu}(\xi)| \lesssim_N 2^{-\ell n/2} \mathbf{1}_{|\xi|\approx 2^{\ell-j}} 
+ 2^{-\ell N} 2^j |\xi| \mathbf{1}_{|\xi|\lesssim 2^{\ell-j}} + (2^j |\xi|)^{-N} \mathbf{1}_{|\xi|\gtrsim 2^{\ell-j}}.
\end{equation}
If $\ell<0$, then for all $N\ge 0$,
\begin{equation}\label{eqn:mftpf-4.1}
	|\Phi^\flat_{j,\mu}(\xi)| + 2^{-2dj} |\partial_\mu \Phi^\flat_{j,\mu}(\xi)|\lesssim_N \min( 2^j|\xi|, (2^j |\xi|)^{-N} ).
\end{equation}
These estimates imply that $A\lesssim 1$.
Both estimates follow from direct computation using the definitions. We indicate some of the details for the term $\Phi^\flat_{j,\mu}(\xi)$.
The term $2^{-2dj}\partial_\mu \Phi^\flat_{j,\mu}(\xi)$ satisfies the same estimates and is handled in the same way.
First suppose that $\ell\ge 0$.
If $|\xi|\approx 2^{\ell-j}$, we use van der Corput's lemma which gives $|\Phi_{j,\mu}(\xi)|\lesssim 2^{-\ell n/2}$.
Moreover, $|\Phi_{j,\mu}(0)|\lesssim_N 2^{-\ell N}$ for all $N\ge 0$ using integration by parts.
If $|\xi|\gtrsim 2^{\ell-j}$, then integration by parts gives $|\Phi_{j,\mu}(\xi)|\lesssim_N (2^j |\xi|)^{-N}$ for all $N\ge 0$. 
Also using rapid decay of the Schwartz function $\widehat{\vartheta_{j-\ell}}$ we obtain the claim.
If $|\xi|\lesssim 2^{\ell-j}$, we use $\Phi^\flat_{j,\mu}(0)=0$, the mean value theorem 
and the estimate $|\nabla \Phi^\flat_{j,\mu}(\xi)|\lesssim_N 2^j 2^{-\ell N}$.
It remains to consider \eqref{eqn:mftpf-4.1}. Suppose that $\ell < 0$.
If $|\xi|\lesssim 2^{-j}$, we argue in the same way as before using the mean value theorem.
If $|\xi|\gtrsim 2^{-j}$, integration by parts and rapid decay of $\widehat{\vartheta_j}$ give $|\Phi^\flat_{j,\mu}(\xi)|\lesssim_N (2^j |\xi|)^{-N}$ as required.

\subsection{The low frequency case\texorpdfstring{: $\ell\in \mathcal{L}_3$}{}}
The third summand in \eqref{eqn:mftpf-2} can be written as 
\begin{equation}
	\Big\|\sup_{\alpha\in\mathcal{A}_s} \sup_{|\mu|\le 2^{-C_0 s}} \Big|\sum_{J_- \le j\le J_{+,\mu}} \sum_{\beta \in \mathcal{B}_s(\alpha)\cap [0,1)^n} S(\alpha,\beta) e(x\cdot\beta) \mathcal{F}_{\R^n}^{-1}[\Phi_{j,\mu}\cdot \chi_s]* M_{-\beta} f(x)\Big|\Big\|_{\ell^2_x(\Z^n)},
\end{equation}
where $J_-=\varepsilon_1^{-1}s$ and $J_{+,\mu}$ is the largest integer $j$ so that $\mu 2^{2dj}<2^{-C_1 s+1}$.
Using the triangle inequality, the previous display is
\begin{equation}\label{eqn:mftpf-5.1}
\lesssim	2^{(n+1)s} \|\sup_{|\mu|\le 2^{-C_0 s}} \tau_\mu* |f| \|_{\ell^2_x(\Z^n)}
	+ \Big\|\sup_{\alpha\in\mathcal{A}_s} \sup_{J\ge J_-} 
	\Big|\fm{\mathscr{L}_{s,\alpha}\Big[\sum_{J_-\le j\le J} \widehat{K_j} \Big]}f\Big|
	\Big\|_{\ell^2(\Z^n)},
\end{equation}
where  $\tau_\mu = |\widetilde{\tau}_\mu*_{\R^n}\mathcal{F}^{-1}_{\R^n}(\chi_s)|$ and
\[ \widetilde{\tau}_\mu(y) = \sum_{J_-\le j\le J_{+,\mu}} \Big(\mathcal{F}_{\R^n}^{-1}[\Phi_{j,\mu}](y) - K_j(-y) \Big) 
= \sum_{J_-\le j\le J_{+,\mu}} (e(\mu|y|^{2d})-1) K_j(-y). \]
Using the definition of $J_{+,\mu}$ we obtain $\tau_\mu * |f| \lesssim 2^{-C_1 s} M_{\mathrm{HL}} f $, where $M_{\mathrm{HL}}$ denotes the discrete Hardy--Littlewood maximal function.
Choosing $C_1>n+1$, this takes care of the first term in \eqref{eqn:mftpf-5.1}, while the second term 
is handled by an application of Lemma \ref{lem:mf-truncsingint} (setting $\mathcal{K}_j=0$ for $0\le j<J_-$).

\subsection{Proof of Lemma \ref{lem:basic}}\label{sec:lemma1pf}
	By \eqref{eqn:Lfact},
	\[ \fm{\mathscr{L}_{s,\alpha}[m_\nu]}f = \fm{\mathscr{L}_{s,\alpha}[1]} \Big( \fm{\mathscr{L}^\sharp_{s}[m_\nu]} f \Big). \] 
	From Proposition \ref{prop:weylsums} we obtain $\gamma>0$ so that for every $\nu\in\mathcal{I}$ and $s\ge 1$,
	\[ \| \sup_{\alpha\in \R} | \fm{\mathscr{L}_{s,\alpha}[m_\nu]}f(x) | \|_{\ell_{x}^2(\Z^n)} \lesssim_{d,n} 2^{-\gamma s} \| \fm{\mathscr{L}^\sharp_{s}[m_\nu]} f \|_{\ell^2(\Z^n)}. \]
	Using \eqref{eqn:Lsmshdef} and Parseval's identity,
	\[ \| \fm{\mathscr{L}^\sharp_{s}[m_\nu]} f \|^2_{\ell^2(\Z^n)} = \int_{[0,1]^n} \Big|\sum_{\beta\in\mathcal{B}^\sharp_s} m_\nu(\xi-\beta) \widetilde{\chi}_s(\xi-\beta) \widehat{f}(\xi)\Big|^2 d\xi. \]
	Summing over $\nu\in\mathcal{I}$ and using disjointness of the supports of the functions $\widetilde{\chi}_s(\cdot - \beta)$ for different $\beta$ we have
	\[ \| \fm{\mathscr{L}^\sharp_{s}[m_\nu]} f(x) \|^2_{\ell_{x,\nu}^2(\Z^n\times\mathcal{I})} = \int_{[0,1]^n} \sum_{\nu\in\mathcal{I}} \sum_{\beta\in\mathcal{B}^\sharp_s} |m_\nu(\xi-\beta) \widetilde{\chi}_s(\xi-\beta) \widehat{f}(\xi)|^2 \, d\xi \]
	\[ \le \Big( \sup_{\xi\in [0,1]^n} \sum_{\nu\in\mathcal{I}} \sum_{\beta\in\mathcal{B}^\sharp_s} |m_\nu(\xi-\beta) \widetilde{\chi}_s(\xi-\beta)|^2 \Big) \|f\|_{\ell^2(\Z^n)}^2, \]
	which by definition of $\widetilde{\chi}_s$ (and disjointness again) is
	\[ \le  \Big( \sup_{|\xi|\le 1} \sum_{\nu\in\mathcal{I}} |m_\nu(\xi)|^2 \Big) \|f\|_{\ell^2(\Z^n)}^2. \]

\subsection{Proof of Lemma \ref{lem:mf-truncsingint}}\label{sec:lemma2pf}
	From the assumptions we have the standard estimate
	\begin{equation}\label{eqn:Kjftest}
	|\widehat{\mathcal{K}_j}(\xi)| \lesssim \min(2^j |\xi|, (2^{j}|\xi|)^{-1})
	\end{equation}
	for all $\xi\in\R^n, j\ge 1$.
	We distinguish two cases: either $J\le 2^{C s}$ or $J>2^{C s}$. Here $C$ is a large constant (to be determined).
	
	{\em Case I:} $J\le 2^{C s}$. 
	Using the numerical inequality \eqref{eqn:rm}, we bound  $|\fm{\mathscr{L}_{s,\alpha}[\widehat{\mathcal{K}^{0,J}}]}f(x)|$
	by
	\[ \sqrt{2}\sum_{l\le C s} \Big( \sum_{\kappa\le 2^{C s-l}} |\fm{\mathscr{L}_{s,\alpha}[\widehat{\mathcal{K}^{\kappa 2^l, (\kappa+1)2^l}}]}f(x)|^2\Big)^{1/2} \]
	plus
	\[ |\fm{\mathscr{L}_{s,\alpha}[\widehat{\mathcal{K}_1}]}f(x)|. \]
	By Lemma \ref{lem:basic} the maximal operator associated with the second term has $\ell^2\to\ell^2$ operator norm $\lesssim 2^{-\gamma s}$ as required.
	Hence, we are left with having to estimate
	\[ \Big(\sum_{\kappa\le 2^{C s-l}} \|\sup_{\alpha\in\mathcal{A}_s}|\fm{\mathscr{L}_{s,\alpha}[\widehat{\mathcal{K}^{\kappa 2^l, (\kappa+1)2^l}}]}f(x)|\|_{\ell^2_x(\Z^n)}^2 \Big)^{1/2} \]
	for every fixed $l\le Cs$. By Lemma \ref{lem:basic} this quantity is
	\[ \lesssim 2^{-\gamma s} \sup_{|\xi|\le 1}\Big(  \sum_{\kappa\le 2^{C s-l}} |\widehat{\mathcal{K}^{\kappa 2^l, (\kappa+1)2^l}}(\xi)|^2 \Big)^{1/2} \|f\|_{\ell^2(\Z^n)}, \]
	which in turn is $\lesssim 2^{-\gamma s} \|f\|_{\ell^2(\Z^n)}$ by \eqref{eqn:Kjftest}.
	This concludes the proof for Case I.
	
	{\em Case II:} $J>2^{C s}$. Here we will exploit the enormous magnitude of the spatial scale $J$. 
	The argument is vaguely similar to the proof of Cotlar's inequality for maximally truncated Calder\'on--Zygmund operators.
	By subtracting a Calder\'{o}n--Zygmund operator 
	(leading to a term that no longer involves a supremum over $J$ and is handled directly by Lemma \ref{lem:basic}) it suffices to
	consider
	\begin{equation}\label{eqn:mftpf-5}
		\|\sup_{\alpha\in\mathcal{A}_s} \sup_{J\ge 2^{Cs}} |\fm{\mathscr{L}_{s,\alpha}[\widehat{\mathcal{K}^{J,\infty}}]}f | \|_{\ell^2(\Z^n)}.
	\end{equation}
	Let $\varphi$ denote a non-negative Schwartz function on $\R^n$ with $\mathrm{supp}\,\widehat{\varphi}\subset \{|\xi|\le 1/2\}$ and $\int \varphi=1$. Write $\varphi_j(x)=2^{-jn}\varphi(2^{-j} x)$.
	We claim that it suffices to estimate
	\begin{equation}\label{eqn:mftpf-6}
		\|\sup_{\alpha\in\mathcal{A}_s} \sup_{J\ge 2^{Cs}} |\fm{\mathscr{L}_{s,\alpha}[\widehat{\varphi_J}\,\widehat{\mathcal{K}}]}f | \|_{\ell^2(\Z^n)}.
	\end{equation}
	where $\mathcal{K}=\mathcal{K}^{0,\infty}$.
	To see this we consider
	\begin{equation}\label{eqn:mftpf-6-1}
		\|\sup_{\alpha\in\mathcal{A}_s} \sup_{J\ge 2^{Cs}} |\fm{\mathscr{L}_{s,\alpha}[\widehat{\mathcal{K}^{J,\infty}} -\widehat{\varphi_J}\,\widehat{\mathcal{K}}]}f | \|_{\ell^2(\Z^n)}.
	\end{equation}
	This is majorized by
	\[ 
		\|\sup_{\alpha\in\mathcal{A}_s} |\fm{\mathscr{L}_{s,\alpha}[\widehat{\mathcal{K}^{J,\infty}} -\widehat{\varphi_J}\,\widehat{\mathcal{K}}]}f | \|_{\ell_{x,J}^2(\Z^n\times \N_0)},
	\]
	which by Lemma \ref{lem:basic} is
	\[ 
	\lesssim 2^{-\gamma s} \sup_{|\xi|\le 1} \Big( \sum_{J\ge 1} | \widehat{\mathcal{K}^{J,\infty}}(\xi) -\widehat{\varphi_J}(\xi)\,\widehat{\mathcal{K}}(\xi)  |^2   \Big)^{1/2} \|f\|_{\ell^2(\Z^n)}.
	\]
	From \eqref{eqn:Kjftest} and the definition of $\varphi$ one derives
	\[ |\widehat{\mathcal{K}^{J,\infty}}(\xi) -\widehat{\varphi_J}(\xi)\,\widehat{\mathcal{K}}(\xi)| \lesssim \min(2^J |\xi|, (2^J |\xi|)^{-1}), \]
	which concludes the proof of the claim.
	Thus it remains to bound \eqref{eqn:mftpf-6}.
	Using \eqref{eqn:Lsmspace} and band-limitedness of $\varphi$ we write
	\[ \fm{\mathscr{L}_{s,\alpha}[\widehat{\varphi_J}\,\widehat{\mathcal{K}}]}f(x) = \sum_{\beta \in \mathcal{B}_s(\alpha)\cap [0,1)^n} S(\alpha,\beta) e(x\cdot\beta) (\varphi_J * \mathscr{K}_s * M_{-\beta} f) (x), \]
	where the convolutions are in $\Z^n$ and $\mathscr{K}_s = \mathcal{F}^{-1}_{\R^n}( \widehat{\mathcal{K}}\cdot \chi_s)$.
	
	Denote the least common multiple of all integers in the interval $[2^{s-1}, 2^s]$ by $Q_s$. 
	By choosing $C$ large enough we may achieve that if $J\ge 2^{C s}$, then
	\[ 2^J \ge Q_s^{100n}. \]
	Then $Q_s$ is small compared with the spatial scale of $\varphi_{J}$ and this allows us to compare 
	\eqref{eqn:mftpf-6} favourably with the averaged version
	\begin{equation}\label{eqn:mftpf-7}
		\Big(\frac1{Q_s^n} \sum_{u\in [Q_s]^n} \Big\|\sup_{\alpha\in\mathcal{A}_s} \sup_{j\ge 1} \Big|\sum_{\beta \in \mathcal{B}_s(\alpha)\cap [0,1)^n} S(\alpha,\beta) e(x\cdot\beta) (\varphi_J * \mathscr{K}_s * M_{-\beta} f) (x-u)\Big|\Big\|_{\ell^2_x(\Z^n)}^2 \Big)^{1/2}.
	\end{equation}
	The resulting difference term can be handled by the mean value theorem and using that $2^J\ge Q_s^{100n}$.
	More specifically, given $u\in [Q_s]^n$ we brutally estimate
	\[
		\Big\|\sup_{\alpha\in\mathcal{A}_s} \sup_{J\ge 2^{Cs}} \Big|\sum_{\beta \in \mathcal{B}_s(\alpha)\cap [0,1)^n} S(\alpha,\beta) e(x\cdot\beta) \Big((\varphi_J * \mathscr{K}_s * M_{-\beta} f) (x) - (\varphi_J * \mathscr{K}_s * M_{-\beta} f) (x-u)\Big) \Big|\Big\|_{\ell^2_x(\Z^n)}.
	\]
	\[
		\lesssim 2^{(n+1)s} \sup_{\beta\in [0,1)^n} \|\sup_{J\ge 2^{Cs}} |(\varphi_J - \varphi_J(\cdot - u)) * \mathscr{K}_s * M_{-\beta} f|\|_{\ell^2(\Z^n)}.
	\]
	Since $|u|2^{-J} \le Q_s^{-100n+1}\le 2^{-(100n+1)s}$, the previous display is (say)
	\[
	\lesssim 2^{-10s} \sup_{\beta\in [0,1)^n} \|M_{\mathrm{HL}} ( \mathscr{K}_s * M_{-\beta} f )\|_{\ell^2(\Z^n)}\lesssim 2^{-10s} \|f\|_{\ell^2(\Z^n)},
	\]
	where $M_{\mathrm{HL}}$ denotes the discrete Hardy--Littlewood maximal function.

	It now remains to estimate \eqref{eqn:mftpf-7}.
	Expanding the $\ell^2$ norm and changing variables $x\mapsto x+u$ we write the square of \eqref{eqn:mftpf-7} as
	\begin{equation}\label{eqn:mftpf-8}
	\frac1{Q_s^n} \sum_{u\in [Q_s]^n} \sum_{x\in\Z^n} \sup_{\alpha\in\mathcal{A}_s} \sup_{j\ge 1} \Big|\sum_{\beta \in \mathcal{B}_s(\alpha)\cap [0,1)^n} S(\alpha,\beta) e((x+u)\cdot\beta) (\varphi_J * \mathscr{K}_s * M_{-\beta} f) (x)\Big|^2,
	\end{equation}
	Changing variables $u\mapsto v-x$ and using periodicity ($Q_s$ is divisible by the denominator of $\beta$ for all $\beta\in\mathcal{B}^\sharp_s$) this becomes
	\begin{equation}\label{eqn:mftpf-9}
	\frac1{Q_s^n} \sum_{v\in [Q_s]^n} \sum_{x\in\Z^n} \sup_{\alpha\in\mathcal{A}_s} \sup_{j\ge 1} \Big|\sum_{\beta \in \mathcal{B}_s(\alpha)\cap [0,1)^n} S(\alpha,\beta) e(v\cdot\beta) (\varphi_J * \mathscr{K}_s * M_{-\beta} f) (x)\Big|^2.
	\end{equation}
	Having decoupled $v$ and $x$ we arrive at the pointwise estimate
	\[ \sup_{\alpha\in\mathcal{A}_s} \sup_{j\ge 1} 
	\Big|\varphi_J * \sum_{\beta \in \mathcal{B}_s(\alpha)\cap [0,1)^n} 
	S(\alpha,\beta) e(v\cdot\beta) (\mathscr{K}_s * M_{-\beta} f) \Big| \]
	\[ \lesssim M_{\mathrm{HL}} \Big( \sup_{\alpha\in\mathcal{A}_s} \Big|\sum_{\beta \in \mathcal{B}_s(\alpha)\cap [0,1)^n} S(\alpha,\beta) e(v\cdot\beta) (\mathscr{K}_s * M_{-\beta} f)\Big| \Big). \]
	Applying the $\ell^2$ bound for $M_{\mathrm{HL}}$, changing variables back $v\mapsto u+x$ and using Lemma \ref{lem:basic} finishes the estimate.

\newcommand{\etalchar}[1]{$^{#1}$}

\end{document}